\documentclass[11pt,a4paper]{amsart}
\usepackage{amssymb,amsmath}
\usepackage{graphicx}
\usepackage{bm,bbm}
\usepackage{color}
\usepackage{tikz}
\usepackage{ifthen}
\usetikzlibrary{snakes}
\usetikzlibrary{patterns}
\usepackage{algorithm}
\usepackage{setspace}
\usepackage[noend]{algpseudocode}
\usepackage{pgfplots}
\usepackage{enumitem}
\usepackage{array,blkarray}
\usepackage{bigdelim}
\usepackage{listings} 
\usepackage[hidelinks]{hyperref}
\usepackage{amssymb}
\usepackage{isotope}
\usepackage{stmaryrd}
\usepackage[]{placeins}
\usepackage[normalem]{ulem}
\newcolumntype{L}[1]{>{\raggedright\let\newline\\\arraybackslash\hspace{0pt}}m{#1}}
\newcolumntype{C}[1]{>{\centering\let\newline\\\arraybackslash\hspace{0pt}}m{#1}}
\newcolumntype{R}[1]{>{\raggedleft\let\newline\\\arraybackslash\hspace{0pt}}m{#1}}
\newtheorem{theorem}{Theorem}
\newtheorem{proposition}[theorem]{Proposition}

\theoremstyle{definition}
\newtheorem{definition}[theorem]{Definition}

\newtheorem{lemma}[theorem]{Lemma}

\theoremstyle{remark}
\newtheorem{remark}[theorem]{Remark}

\newtheorem{assumption}[theorem]{Assumption}
\numberwithin{theorem}{section}
\numberwithin{equation}{section}
\numberwithin{table}{section}
\numberwithin{figure}{section}
\definecolor{myBlue}{RGB}{113,104,238} % medium slate blue	
\definecolor{myGreen}{RGB}{154,205,50} % olive Drap
\definecolor{myGreen2}{RGB}{114,175,30} 
\definecolor{myRed}{RGB}{180,50,50}  
\definecolor{myOrange}{RGB}{225,92,22} 
\definecolor{lgray}{RGB}{200,200,200} % lightgray
\definecolor{llgray}{RGB}{155,155,155} % lightgray
\definecolor{mycolor1}{rgb}{0.00000,0.44700,0.74100}%
\definecolor{mycolor2}{rgb}{0.85000,0.32500,0.09800}%
\definecolor{mycolor3}{rgb}{0.92900,0.69400,0.12500}%
\definecolor{mycolor4}{rgb}{0.49400,0.18400,0.55600}%
\definecolor{mycolor5}{rgb}{0.46600,0.67400,0.18800}%
\definecolor{mycolor6}{rgb}{0.30100,0.74500,0.93300}%
\definecolor{mycolor7}{rgb}{0.63500,0.07800,0.18400}%
\newcommand\N{\mathbb N}

\newcommand\R{\mathbb R}

\DeclareMathOperator{\diag}{diag}

\DeclareMathOperator{\tol}{tol}
\DeclareMathOperator{\tr}{tr}
\DeclareMathOperator{\trace}{tr}

\DeclareMathOperator{\grad}{grad}

\DeclareMathOperator{\qf}{qf}

\DeclareMathOperator*{\argmin}{arg\,min}

\newcommand{\Stiefel}{\mathrm{St}(N,V)}

\newcommand\calA{\mathcal A}
\newcommand\calB{\mathcal B}

\newcommand\calE{\mathcal E}

\newcommand\calR{\mathcal R}

\newcommand\calS{\mathcal S}

\newcommand{\etabf}{{\bm \eta}}
\newcommand{\xibf}{{\bm \xi}}
\newcommand{\phibf}{{\bm \phi}}
\newcommand{\zetabf}{{\bm \zeta}}
\newcommand{\varphibf}{{\bm \varphi}}
\newcommand{\psibf}{{\bm \psi}}

\newcommand{\ybf}{\bm{y}}
\newcommand{\qbf}{\bm{q}}
\newcommand{\sbf}{\bm{s}}
\newcommand{\ubf}{\bm{u}}
\newcommand{\vbf}{\bm{v}}
\newcommand{\wbf}{\bm{w}}
\newcommand{\zbf}{\bm{z}}

\def\dr{\,\text{d}r}
\def\drp{\,\text{d}r'}

\def\dt{\,\text{d}t}
\def\dx{\,\text{d}x}

\newcommand{\Onebf}{\bm{I}_{N}}
\newcommand{\Nullbf}{\bm{0}_{N}}

\newcommand{\out}[2]{\llbracket{#1},{#2}\rrbracket_H}
\begin{document}
\title[Energy-adaptive Riemannian Optimization]{Energy-adaptive Riemannian Optimization\\ on the Stiefel Manifold}
\author[]{R.~Altmann$^*$, D.~Peterseim$^{\dagger}$, T.~Stykel$^{\dagger}$}
\address{${}^{*}$ Institute of Mathematics, University of Augsburg, Universit\"atsstr.~12a, 86159 Augsburg, Germany}
\address{${}^{\dagger}$ Institute of Mathematics \& Centre for Advanced Analytics and Predictive Sciences (CAAPS), University of Augsburg, Universit\"atsstr.~12a, 86159 Augsburg, Germany}
\email{\{robert.altmann, daniel.peterseim, tatjana.stykel\}@uni-a.de}
\thanks{The work of Daniel Peterseim is part of a project that has received funding from the European Research Council (ERC) under the European Union's Horizon 2020 research and innovation programme (Grant agreement No.~865751 --  RandomMultiScales).}
\date{\today}
%
%
%=============================================================================
%=========  Abstract
%=============================================================================
\begin{abstract}
This paper addresses the numerical solution of nonlinear eigenvector problems such as the Gross-Pitaevskii and Kohn-Sham equation arising in computational physics and chemistry. These problems characterize critical points of energy minimization problems on the infinite-dimensional Stiefel manifold. To efficiently compute minimizers, we propose a novel Riemannian gradient descent method induced by an energy-adaptive metric. Quantified convergence of the methods is established under suitable assumptions on the underlying problem. A non-monotone line search and the inexact evaluation of Riemannian gradients substantially improve the overall efficiency of the method. Numerical experiments illustrate the performance of the method and demonstrates its competitiveness with well-established schemes. 
\end{abstract}
%
%
%=============================================================================
%=========  Title / Contents
%=============================================================================
\maketitle
%\setcounter{tocdepth}{3}
%\tableofcontents
%
%
%81Q10: Quantum theory -- Mathematical theory -- Selfadjoint op, spectral analysis
%65N25: Numerical analysis -- PDEs, BVPs -- Eigenvalue probems
%
%=============================================================================
%=========  Introduction
%=============================================================================
%		
\section{Introduction}
This paper is devoted to the numerical solution of energy minimization problems stated on the infinite-dimensional Stiefel manifold of index $N$ containing~$N$-tuples of $L^2$-ortho\-nor\-mal functions. The Kohn-Sham model~\cite{HohK64,KohS65,LeB05} is a prototypical example. In this popular model from {\em density functional theory} in computational chemistry, the state of the system is described by $N>1$ functions (orbitals), which need to satisfy $L^2$-orthogonality conditions. The ground state of the system minimizes the Kohn-Sham energy under these orthogonality constraints, i.e.,~on the Stiefel manifold of index $N$, cf.~\cite{YanMLW09}. 
For $N=1$, the Stiefel manifold boils down to the unit sphere in $L^2$. In this special case, the Gross-Pitaevskii model for Bose-Einstein condensates of ultracold bosonic gases~\cite{LSY01,PiS03} is a relevant example. Its ground state is the global minimizer of the corresponding Gross-Pitaevskii energy functional on the Stiefel manifold which simply represents a unit mass constraint. 

More generally, the ground states of energy functionals on the Stiefel manifold as well as further critical points are characterized by coupled systems of eigenvalue problems of partial differential equations (PDEs) with eigenvector nonlinearities, so-called {\em nonlinear eigenvector problems}. Existing approximation methods for these problems are either linked to linear eigenvalue solvers or to Riemannian optimization. A well-known iteration scheme for the nonlinear eigenvector problem is the {\em self-consistent field iteration} (SCF). Each SFC iteration step involves the solution of a {\em linear} eigenvalue problem, see, e.g., \cite{CanL00,Can01,doi:10.1137/20M1332864} and~\cite{JarU21ppt} for its connection to Newton's method. On the Riemmanian side,  the {\em direct constrained minimization algorithm} (DCM) is very popular. DCM results from a standard minimization approach~\cite{YanMW06,AloA09,SchRNB09} and is based on the Riemannian gradient descent method in~$L^2$. However, this method requires suitable preconditioning to work. In the special case of the Gross-Pitaevskii eigenvalue problem, the DCM is known as the {\em discrete normalized gradient flow} \cite{BaD04}. Although empirically successful, the preconditioning or stable time discretization comes with the drawback of deviating from the gradient descent structure. In this case, the energy decay cannot be guaranteed anymore. 
In~\cite{HenP20}, an alternative Riemannian gradient descent scheme was proposed for the special case of the Gross-Pitaevskii problem, which is based on a~gradient flow defined in an~energy-adaptive metric. The resulting method is convergent and energy diminishing for sufficiently small step sizes. The energy diminishing property even gives rise to global convergence to the ground state~\cite{HenP20} and turns out to be valuable in the context of reliable a~ posteriori error control~\cite{HEID2021110165}. 

In this paper, we generalize this promising yet simple energy-adaptive Riemannian descent method to nonlinear eigenvector problems formulated on the Stiefel manifold. The general functional analytical setting of the considered problems is presented in Section~\ref{sect:generalModel}. 
Details on the infinite-dimensional Stiefel manifold, its tangent and normal spaces, and the orthogonal projection onto the tangent space are then discussed in Section~\ref{sec:Stiefel}. Therein, we show that the mentioned projection can be characterized by a saddle point problem, which facilitates the proposed algorithm significantly. 
Finally, several retractions are introduced, which are needed to transform tangent vectors back to the manifold. Section~\ref{sect:Riemannian} presents the novel energy-adaptive Riemannian gradient descent method. Its convergence analysis generalizes the approach of~\cite{Zha21ppt} for $N=1$. It is independent of the space dimension and, hence, also independent of possible spatial discretization by finite elements, spectral methods or related schemes. The convergence is further accelerated by the non-monotone line search algorithm of ~\cite{ZhaW04,WenY13}. Moreover, we identify a connection to a preconditioned version of DCM, which motivates the substantial reduction of the computational complexity of the new method related on inexact gradient computations. 
In Section~\ref{sect:examples}, we show that the Gross-Pitaevskii and Kohn-Sham models fit into the given framework. Numerical experiments for the Kohn-Sham model illustrate the performance of the presented method. Using the step size control and suitable inexact gradient computations prove the new approach competitive with established methods such as SCF and DCM. 
%
%
%=============================================================================
%=========  General model
%=============================================================================
\section{Energy Minimization Problem on the Stiefel Manifold}\label{sect:generalModel}
This section introduces an abstract constrained PDE energy minimization problem and its connection to a coupled system of nonlinear eigenvector problems formulated on the infinite-dimensional Stiefel manifold of index~$N$. Particular examples such as the Gross-Pitaevskii eigenvalue problem and the Kohn-Sham model will be discussed in detail in Section~\ref{sect:examples}. 
%
%
%=============================================================================
%
\subsection{Spaces and bilinear forms}
We consider a~space $\tilde{V}\subseteq H^1(\Omega)$ for a~given domain $\Omega\subseteq \R^d$ and define 
\[
  V:=\tilde{V}^N, \qquad
  H := [L^2(\Omega)]^N
\]
with $N\ge1$. The suitable choice of the Hilbert space~$\tilde{V}$ depends on the particular application, cf.~the examples in Section~\ref{sect:examples}. Let $V^*$ denote the dual space of $V$. We assume that~$V\subset H\subset V^*$ form a~Gelfand triple~\cite[Ch.~23.4]{Zei90a}. 
Throughout this paper, we use the row-vector notation for $N$-frames, i.e., we write $\vbf = (v_1, \dots, v_N) \in V$. This allows us to adapt the notion of typical matrix-vector multiplication, i.e., we may multiply $\vbf$ by an $N\times N$ matrix from the right, leading again to an element of $V$. Furthermore, for $\vbf,\wbf\in H$, we define the~dot product $\vbf\cdot\wbf := \sum_{j=1}^N v_j w_j$. We say that the components of $\vbf\in V\setminus\{\bm{0}\}$ are linearly independent,
if there is no non-zero vector $x\in\mathbb{R}^N$ such that $\vbf x=0$.

% bilinear forms
On the pivot space $H$, we introduce an~outer product~$\out{\,\cdot\,}{\cdot\,}\colon H\times H\to \R^{N\times N}$ and an~inner product $(\,\cdot\,, \cdot\,)_H\colon H\times H\to \R$. More precisely, for $\vbf, \wbf \in H$, we define   
\begin{equation}\label{eq:outer}
  \out{\vbf}{\wbf} 
  := \begin{bmatrix}
  (v_1, w_1)_{L^2(\Omega)} & \dots & (v_1, w_N)_{L^2(\Omega)} \\
  \vdots & \ddots & \vdots \\ 
  (v_N, w_1)_{L^2(\Omega)} & \dots & (v_N, w_N)_{L^2(\Omega)}
  \end{bmatrix}
  \in \R^{N\times N}
\end{equation}
and
\begin{equation}\label{eq:inner}
  (\vbf, \wbf)_H 
  := \sum_{j=1}^N (v_j, w_j)_{L^2(\Omega)} 
  = \trace\, \out{\vbf}{\wbf},
\end{equation}
where $\trace$ denotes the trace of a~matrix. 
The inner product \eqref{eq:inner} induces the norm \linebreak $\|\vbf\|_H=\sqrt{(\vbf,\vbf)_H}$ on $H$. Some properties of the~outer product
\eqref{eq:outer} are collected in the following lemma, which follows from straight-forward calculations.
\begin{lemma}
\label{lem:propOuterProduct}
Consider~$\vbf, \wbf \in H$ and an arbitrary matrix $S\in\R^{N\times N}$. Then it holds that 
\[
  \out{\vbf}{\wbf S} = \out{\vbf}{\wbf} S, \qquad
  \out{\vbf S}{\wbf} = S^T \out{\vbf}{\wbf}, \qquad
  \out{\vbf}{\wbf} = \out{\wbf}{\vbf}^T. 
\]
\end{lemma} 
For the definition of the energy in the next subsection, we further introduce a (prob\-lem-dependent) bilinear form~$a_\phibf\colon V\times V\to \R$ for a~fixed $\phibf\in V$. With the density function $\rho(\phibf)=\phibf\cdot\phibf$, we consider 
\begin{equation}\label{eq:aphi}
  a_\phibf(\vbf,\wbf)
  = a_0(\vbf,\wbf) + \int_{\Omega} \gamma(\rho(\phibf))\, \vbf\cdot \wbf \dr
  = \sum_{j=1}^N \tilde a_\phibf(v_j, w_j) 
\end{equation}
for $\vbf,\wbf\in V$. Here, $a_0:V\times V\to\R$ is a bilinear form,  which is independent of~$\phibf$, and $\gamma:\R\to\R$ is a continuous nonlinear function with $\gamma(0)=0$. Later, $a_0$ and the term with~$\gamma$ will correspond, respectively, to the quadratic part and the nonlinear part of the energy. 
Note that~\eqref{eq:aphi} encodes a special structure, i.e., $a_\phibf$ can be written as a~sum with a~bilinear form~$\tilde a_\phibf\colon \tilde{V}\times \tilde{V}\to \R$. Within the abstract setting, we consider the following assumption. 
\begin{assumption}[Bilinear form~$\tilde{a}_\phibf$]
\label{ass:aPhi}
For a~fixed $\phibf\in V$, $\tilde{a}_\phibf$ from~\eqref{eq:aphi} is a~symmetric, bounded, and coercive bilinear form on~$\tilde{V}$. 
\end{assumption}
By equation~\eqref{eq:aphi}, the bilinear from~$a_\phibf$ inherits the inner product structure from $\tilde{a}_\phibf$, meaning that $a_\phibf$ is symmetric, bounded, and coercive on~$V$. Thus, it defines an~inner product on $V$ which induces the norm 
$$
\|\vbf\|_{a_\phibf}=\sqrt{a_\phibf(\vbf,\vbf)}, \qquad \vbf\in V.
$$ 
The assumed Gelfand structure implies the existence of a~constant $C_H>0$ such that \mbox{$\|\vbf \|_H \le C_H\, \|\vbf \|_{a_{0}}$}. Moreover, for a bounded~$\phibf\in V \cap [L^{\infty}(\Omega)]^N$, there exists a~constant \mbox{$c_E>0$} such that 
$$ 
c_E\,\|\vbf \|_{a_{\phibf}} \le\|\vbf \|_{a_{0}} \le \|\vbf \|_{a_{\phibf}}\qquad\text{ for all }\vbf\in V.
$$
 The corresponding operator formulation of the bilinear form~$a_\phibf$ reads
\begin{equation}\label{eq:Aphi}
   \langle \calA_{\phibf} \vbf, \wbf\rangle := a_\phibf(\vbf,\wbf) \qquad\text{ for all } \vbf,\wbf\in V, 
\end{equation}
with a linear operator $\calA_{\phibf}\colon V\to V^*$. Assumption~\ref{ass:aPhi} implies that $\calA_{\phibf}$ is symmetric, bounded, and coercive. Hence, it is invertible (for fixed~$\phibf$). Its inverse satisfies   
\begin{equation}\label{eq:invA}
  a_\phibf(\mathcal A_{\phibf}^{-1}\vbf,\wbf) 
  = (\vbf,\wbf)_H \qquad\text{ for all }\vbf,\wbf\in V.
\end{equation}
Next, we show some useful properties of the matrix $\out{\vbf}{\calA_\phibf^{-1}\vbf}$.
\begin{proposition}\label{prop:symmetr}
Let $\phibf, \vbf\in V$ and let $\calA_\phibf$ be defined as in \eqref{eq:Aphi}. Then, under Assumption~\textup{\ref{ass:aPhi}}, the matrix $\out{\vbf}{\calA_\phibf^{-1}\vbf}\in\mathbb{R}^{N\times N}$ is symmetric positive semidefinite. If, additionally, $\vbf\neq \bm{0}$ and its components are linearly independent, then $\out{\vbf}{\calA_\phibf^{-1}\vbf}$ is positive definite. 
\end{proposition}
\begin{proof}
Due to the additive structure of~\eqref{eq:aphi}, there exists a~symmetric and coercive operator~$\tilde{\calA}_\phibf\colon \tilde{V}\to \tilde{V}^*$ corresponding to the bilinear form~$\tilde a_\phibf$ such that 
$$
  \langle \calA_\phibf\vbf, \wbf\rangle 
  = \sum_{j=1}^N\, \langle \tilde{\calA}_\phibf v_j, w_j\rangle \qquad \text{ for all }
\vbf,\wbf\in V.
$$
Thus, we conclude that~$\calA_\phibf^{-1}\vbf = (\tilde{\calA}^{-1}_\phibf v_1, \dots, \tilde{\calA}^{-1}_\phibf v_N)\in V$. Moreover, since $\tilde \calA_\phibf$ is symmetric, so is its inverse, which implies 
\begin{align*}
  \big( \out{\vbf}{\calA_\phibf^{-1}\vbf} \big)_{ij}
   & = \big( v_i, (\calA_\phibf^{-1}\vbf)_j \big)_{L^2(\Omega)}
     = \big( v_i, \tilde\calA_\phibf^{-1} v_j \big)_{L^2(\Omega)} \\
   & = \big( v_j, \,\tilde\calA_\phibf^{-1} v_i \big)_{L^2(\Omega)}
	   = \big( v_j, (\calA_\phibf^{-1}\vbf)_i \big)_{L^2(\Omega)}
     = \big( \out{\vbf}{\calA_\phibf^{-1}\vbf} \big)_{ji}. 
\end{align*}
Further, for an arbitrary vector $x\in\R^N$, we get
\begin{align*}
  x^T \big( \out{\vbf}{\calA_\phibf^{-1}\vbf}\, x\big)
%  = \sum_{i=1}^N \sum_{j=1}^N \big( \out{\vbf}{\calA_\phibf^{-1}\vbf} \big)_{ij} x_i x_j
  & = \sum_{i=1}^N \sum_{j=1}^N \big( v_i, \tilde\calA_\phibf^{-1} v_j \big)_{L^2(\Omega)}\, x_i\, x_j
    = \big( \vbf x, \tilde\calA_\phibf^{-1} (\vbf x) \big)_{L^2(\Omega)} \\
  & = \tilde a_\phibf\big(\tilde\calA_\phibf^{-1} (\vbf x), \tilde\calA_\phibf^{-1} (\vbf x) \big) \geq 0.  
\end{align*}
This shows that $\out{\vbf}{\calA_\phibf^{-1}\vbf}$ is positive semidefinite. Finally, if $\vbf\neq 0$ has linearly independent components, then for all $x\in\R^N\setminus\{0\}$, we have~$\vbf x \neq 0$
%$$ 
%x^T \big( \out{\vbf}{\calA_\phibf^{-1}\vbf}\, x\big)
%=\tilde a_\phibf\big(\tilde\calA_\phibf^{-1} (\vbf x), \tilde\calA_\phibf^{-1} (\vbf x) \big)>0 
%$$
and, hence, $\out{\vbf}{\calA_\phibf^{-1}\vbf}$ is positive definite.
\end{proof}
%
%
%=============================================================================
%
\subsection{Variational form and nonlinear eigenvector problem}\label{sect:generalModel:variationalForm}
Given an index $N\in\N$ and the space $V$, let
\[
\Stiefel
:= \big\{ \phibf\in V\enskip :\enskip \out{\phibf}{\phibf} = \Onebf \big\} 
\]
denote the infinite-dimensional {\em Stiefel manifold} of index $N$. Here, $\Onebf$ is the identity matrix in $\R^{N\times N}$. We will see in Section~\ref{sec:Stiefel} that $\Stiefel$ admits a~structure of an~embedded submanifold of the Hilbert space $V$. Such a~manifold was previously considered in \cite{Usch10,HarM12}. 

This paper is devoted to the abstract constrained energy minimization problem
\begin{equation}\label{eq:opt}
\min_{\phibf\in\Stiefel} \calE(\phibf) 
\end{equation}
with the~energy functional
\begin{equation}\label{eq:energy}
\calE(\phibf)
:= \frac12\, a_0(\phibf,\phibf) + \frac12\, \int_{\Omega} \Gamma(\rho(\phibf)) \dr, \qquad
\Gamma(\rho) 
= \int_0^\rho \gamma(t) \dt.
\end{equation}
Throughout the paper, we make the (physically meaningful) assumption that~$\calE$ is or\-tho\-gonally invariant in the sense that $\calE(\phibf\, Q)=\calE(\phibf)$ for any orthogonal matrix \mbox{$Q\in\mathbb{R}^{N\times N}$}. This means that the energy depends only on the space spanned by the components of $\phibf$ and not on a~particular choice of $\phibf$. This condition is fulfilled in the applications we are interested in, see Section~\ref{sect:examples}.

We are seeking critical points of the energy $\calE$ which represent low-energy states. The state of minimal energy, which is called the {\em ground state}, is of particular interest. 
Critical points of the energy subject to the constraint are characterized by a coupled system of nonlinear eigenvector problems associated with the bilinear form $a_\phibf$ introduced in \eqref{eq:aphi}. The connection follows from the observation that the directional derivative $\mathrm{D}\calE(\phibf)[\vbf]$ of $\calE$ at $\phibf$ along $\vbf$ is given by
\begin{align}
\label{eq:DerivativeE}
\mathrm{D}\calE(\phibf)[\vbf]
= a_\phibf(\phibf,\vbf) 
\qquad\text{for all } \vbf \in V.
\end{align}
The variational formulation of the nonlinear eigenvector problem then reads: seek \linebreak $\phibf\in \Stiefel$ and $N$ eigenvalues $\lambda_1,\dots,\lambda_N\in\R$ such that 
\begin{align}
\label{eq:EVPweakComponents}
\tilde a_\phibf(\phi_j, v_j) 
= \lambda_j\, (\phi_j, v_j)_{L^2(\Omega)}
\qquad\text{ for all }(v_1, \dots, v_N)\in V.
\end{align}
We emphasize that all these problems are coupled, since the bilinear form $\tilde a_\phibf$ contains the information on the entire $N$-frame~$\phibf$. 
%
%
%=============================================================================
%=========  Stiefel
%=============================================================================
\section{Geometry of the Infinite-Dimensional Stiefel Manifold}\label{sec:Stiefel}
In this section, we investigate the geometric structure of the Stiefel manifold $\Stiefel$. First, we state that $\Stiefel$ is an~embedded submanifold of the Hilbert space~$V$. This result can be proved analogously to the finite-dimensional case of the Stiefel matrix manifold; see~\cite[Sect.~3.3.2]{AbsiMS08}. 
\begin{proposition}\label{prop:Stiefel}
The Stiefel manifold $\,\Stiefel$ is a closed embedded submanifold of the Hilbert space $V$. It has co-dimension $N(N+1)/2$.
\end{proposition}
%
%\begin{proof} Let $\calS_{\rm sym}(N)$ denote the set of all real symmetric $N\times N$ matrices.
%Consider the map $F\colon V\to \calS_{\rm sym}(N)$ defined as $F(\phibf)=\out{\phibf}{\phibf} - \Onebf$. Then we have $\Stiefel=F^{-1}(\Nullbf)$, where $\Nullbf$ denotes the zero matrix in $\mathbb{R}^{N\times N}$. This implies that $\Stiefel$ is closed, since it is the preimage of the closed set $\{\Nullbf\}$ for the continuous map $F$.
%
%Next, we show that $F$ is a~submersion or, equivalently, that for all $\phibf\in\Stiefel$, the Fr\'echet derivative $\mathrm{D}F(\phibf)[\,\cdot\,]:V\rightarrow \calS_{\rm sym}(N)$ is surjective. Let $S\in\calS_{\rm sym}(N)$. Then, for $\vbf=\frac{1}{2}\,\phibf\,S\in V$, we obtain
%\begin{align*}
%\mathrm{D} F(\phibf)[\vbf] 
%= \out{\phibf}{\vbf}+\out{\vbf}{\phibf}
%= \tfrac{1}{2}\, \out{\phibf}{\phibf} S+\tfrac{1}{2}\, S^T \out{\phibf}{\phibf}=S.
%\end{align*}
%%
%Thus, by the submersion theorem \cite[Prop.~3.3.3]{AbsiMS08}, $\Stiefel$ is an~embedded submanifold of $V$ and its co-dimension coincides with $\dim(\calS_{\rm sym}(N))=N(N+1)/2$.
%\end{proof}

The {\em tangent space} of $\Stiefel$ at $\phibf\in \Stiefel$ is given by 
\begin{align*}
  T_\phibf\,\Stiefel
  := \big\{ \etabf\in V\enskip :\enskip \out{\etabf}{\phibf} + \out{\phibf}{\etabf} = \Nullbf \big\}. 
 \end{align*}
Hence, $T_\phibf\,\Stiefel$ contains all functions $\etabf\in V$ for which the matrix~$\out{\etabf}{\phibf}$ is skew-symmetric. 
%
%=============================================================================
%
\subsection{Hilbert metric and normal space}
The simplest Riemannian metric on the Stiefel manifold $\Stiefel$ is the {\em Hilbert metric} $g_H$ inherited from the ambient space $V\subset H$. It is given by 
$$
g_H(\etabf,\zetabf)
=(\etabf,\zetabf)_H
=\trace\, \out{\etabf}{\zetabf}
\qquad \text{ for all } \etabf,\zetabf\in T_\phibf\,\Stiefel. 
$$
This metric turns $\Stiefel$ into a~Riemannian submanifold of $V$.  
The {\em normal space at $\phibf\in\Stiefel$ with respect to $g_H$} is then defined as 
$$
\big(T_\phibf\,\Stiefel\big)_H^\perp
=\bigl\{ \zbf\in V \enskip :\enskip g_H(\zbf,\etabf)=0 \;\text{ for all } \etabf\in T_\phibf\,\Stiefel \bigr\}.
$$
The following proposition gives an~explicit characterization of this space. Its proof is similar to the finite-dimensional setting, which can be found in~\cite[Sect.~2.2.1]{EdeAS98}. 
\begin{proposition}
The normal space $(T_\phibf\,\Stiefel)_H^\perp$ at $\phibf\in\Stiefel$ is given by 
\begin{equation}\label{eq:normal}
\big(T_\phibf\,\Stiefel\big)_H^\perp= \bigl\{ \phibf\,S \in V \enskip :\enskip S\in\calS_{\rm sym}(N) \bigr\},
\end{equation}
where $\calS_{\rm sym}(N)$ denotes the set of all real symmetric $N\times N$ matrices.
\end{proposition}
%
%\begin{proof}
%It follows from Proposition~\ref{prop:Stiefel} that $\dim\bigl((T_\phibf\,\Stiefel)_H^\perp\bigr)=N(N+1)/2$. Further, the linear space in the right-hand side of \eqref{eq:normal} has the same dimension. 
%On the other hand, for any element $\phibf\,S$ with $S\in \calS_{\rm sym}(N)$, we have
%$$
%  g_H(\phibf\,S,\etabf) 
%  = \trace\, \out{\phibf\,S}{\etabf} 
%  = \trace(S^T \out{\phibf}{\etabf})
%  = 0
%  \qquad \text{ for all } \etabf\in T_\phibf\,\Stiefel,
%$$
%since $S$ is symmetric and $\out{\phibf}{\etabf}$ is skew-symmetric. 
%Hence, $\phibf\,S\in (T_\phibf\,\Stiefel)_H^\perp$, which readily proves the assertion. 
%\end{proof}
%
We now introduce an~$H$-orthonormal basis of $(T_\phibf\,\Stiefel)_H^\perp$. Let $S^{ij} \in \mathcal{S}_{\rm sym}(N)$ denote the (normalized) symmetric matrix which has a~non-zero entry at positions $(i,j)$ and $(j,i)$ and a zero otherwise. More precisely, we have 
\begin{equation}\label{eq:Sij}
\arraycolsep=2pt
\begin{array}{rcll}
  S^{ii} & = & e_i^{}\, e_i^T, &\qquad 1\leq i\leq N,\\[0.3em]
  S^{ij} & = & \tfrac{1}{\sqrt 2} \big(e_i^{}\, e_j^T + e_j^{}\, e_i^T\big), &\qquad 1\leq i<j\leq N,
\end{array}
\end{equation}
where $e_j$ denotes the $j$th column of $\Onebf$. 
Note that these matrices form a~basis of $\mathcal{S}_{\rm sym}(N)$. For $1\leq i\le j\leq N$, we define the functions $\phibf^{ij} := \phibf\, S^{ij} \in (T_\phibf\,\Stiefel)_H^\perp$. This means that
\begin{equation} \label{eq:phi_ij}
\arraycolsep=2pt
\begin{array}{rcll}
	\phibf^{ii} &=& (0, \dots, 0, \phi_i, 0, \dots, 0), &\quad 1\leq i\leq N,\\[0.3em]
	\phibf^{ij} &=& \tfrac{1}{\sqrt 2}\, (0, \dots, 0, \phi_j, 0, \dots, 0, \phi_i, \dots, 0),  &\quad 1\leq i<j\leq N,
\end{array}
\end{equation}
where $\phi_j$ (the $j$th component of $\phibf$) is placed at the $i$th position and~$\phi_i$ at the $j$th position. Properties of these functions are summarized in the following proposition. 
\begin{proposition}
\label{prop:Phihat}
Let $\phibf\in \Stiefel$. Then the functions $\phibf^{ij}$, $1\leq i\le j\le N$, introduced in~\eqref{eq:phi_ij} form an~$H$-orthonormal basis of $(T_\phibf\,\Stiefel)_H^\perp$. 
\end{proposition}
\begin{proof}
First, we show the $H$-orthonormality of the functions $\phibf^{ij}$. Since~$\phibf\in \Stiefel$, we obtain 
\[
	(\phibf^{ii}, \phibf^{\ell \ell})_H
	= \sum_{m=1}^N (\phi^{ii}_m, \phi^{\ell \ell}_m)_{L^2(\Omega)} 
	= \sum_{m=1}^N \delta_{im}\delta_{\ell m}\, (\phi_i, \phi_\ell)_{L^2(\Omega)} 
	= \delta_{i\ell}.
\] 
For $k<\ell$, we have 
\[
	(\phibf^{ii}, \phibf^{k\ell})_H
	= (\phi_i, \phi^{k\ell}_i)_{L^2(\Omega)} 
	= \tfrac{1}{\sqrt 2}\, \delta_{ik} (\phi_i, \phi_\ell)_{L^2(\Omega)} 
	+ \tfrac{1}{\sqrt 2}\, \delta_{i\ell} (\phi_i, \phi_k)_{L^2(\Omega)} 
	= \sqrt 2\, \delta_{ik} \delta_{i\ell}
	= 0.
\]
Finally, for $i<j$ and $k<\ell$, which implies $\delta_{i\ell}\delta_{j k}=0$, we derive 
\[
	(\phibf^{ij}, \phibf^{k\ell})_H
	= \sum_{m=1}^N (\phi^{ij}_m, \phi^{k\ell}_m)_{L^2(\Omega)} 
	= \frac12\, \Big( \delta_{i k}\delta_{j\ell} 
	+ \delta_{j\ell}\delta_{i k} \Big)
	= \delta_{i k}\delta_{j\ell}. 
\] 
Obviously, the functions $\phibf^{ij}$, $1\leq i\leq j\leq N$, span $(T_\phibf\,\Stiefel)_H^\perp$ and, hence, they form an~$H$-orthonormal basis of $(T_\phibf\,\Stiefel)_H^\perp$.
\end{proof}
%
%
%=============================================================================
%
\subsection{The \texorpdfstring{$a_\phibf$}{}-metric, normal space, and \texorpdfstring{$a_\phibf$}{}-orthogonal projection}
An~alternative Riemannian metric on the Stiefel manifold $\Stiefel$ can be defined by using the inner product $a_\phibf(\,\cdot\,,\cdot\,)$ introduced in \eqref{eq:aphi} as 
$$
g_a(\etabf,\zetabf)=a_\phibf(\etabf,\zetabf)\qquad \text{ for all } \etabf,\zetabf\in T_\phibf\,\Stiefel.
$$
Then the {\em normal space at $\phibf\in\Stiefel$ with respect to $g_a$} is defined as 
$$
(T_\phibf\,\Stiefel)_{a}^\perp=\bigl\{ \zbf\in V \enskip :\enskip g_a(\zbf,\etabf)=0 \;\text{ for all } \etabf\in T_\phibf\,\Stiefel \bigr\}.
$$
Our goal is now to construct a~basis of $(T_\phibf\,\Stiefel)_{a}^\perp$. To this end, we 
introduce the functions~$\psibf^{k\ell} \in V$ for~$1\leq k\le \ell\le N$ as solutions to
\begin{subequations}
\label{property:Psi}
\begin{align}
	a_\phibf(\psibf^{k\ell}, \etabf) 
	&= 0 \hspace{2.8em}\text{ for all } \etabf\in T_\phibf\,\Stiefel, \label{property:Psi:a} \\
	(\psibf^{k\ell}, \phibf^{ij})_H 
	&= \delta_{ik}\delta_{j\ell} \quad\text{ for } 1\le i\le j\le N, \label{property:Psi:b}
\end{align}
\end{subequations}
where $\phibf^{ij}$ are defined in~\eqref{eq:phi_ij}. The following proposition establishes the well-posedness of these problems. 
\begin{proposition}
\label{prop:existencePsi}
There exist unique functions~$\psibf^{k\ell} \in V$, $1\le k\le \ell\le N$, satisfying~\eqref{property:Psi}. 
\end{proposition}
\begin{proof}
Let the indices $1\le k\le \ell\le N$ be arbitrary but fixed. We can write~\eqref{property:Psi} as a~saddle point problem. Hence, we seek for $\psibf^{k\ell} \in V$ and Lagrange multipliers $\mu^{ij}\in \R$, $1\le i\le j\le N$, such that 
\begin{align*}
	a_\phibf(\psibf^{k\ell}, \vbf) + \sum_{i\le j} (\phibf^{ij}, \vbf)_H \mu^{ij}
	&= 0 \hspace{2.8em}\text{ for all } \vbf\in V, \\
	(\psibf^{k\ell}, \phibf^{ij})_H \hspace*{25mm}
	&= \delta_{ik}\delta_{j\ell} \quad\text{ for } 1\le i\le j\le N.
\end{align*}
By Assumption~\ref{ass:aPhi}, the bilinear form $a_\phibf$ is coercive. Moreover, the number of constraints equals~$N(N+1)/2$ and is, hence, finite. In this case, the corresponding inf-sup stability follows from the linear independence of the functions~$\phibf^{ij}$.     
% finte constraints: only need linearity, continuiuty, surjectivity
As a result, \cite[Ch.~III.4]{Bra07} implies the existence of a unique solution $\psibf^{k\ell}\in V$. Note that~$\psibf^{k\ell}$ satisfies~\eqref{property:Psi:a}, since by Proposition~\ref{prop:Phihat}, we have $(\phibf^{ij}, \etabf)_H = 0$ for all $\etabf\in T_\phibf\,\Stiefel$. 
\end{proof}

Next, we characterize the normal space $(T_\phibf\,\Stiefel)_a^\perp$ by providing a~basis of it.

\begin{proposition}
Let $\phibf\in \Stiefel$. Then the functions $\psibf^{k\ell}\in V$, $1\leq k\le \ell\le N$, satisfying~\eqref{property:Psi} form 
a~basis of the normal space $(T_\phibf\,\Stiefel)_a^\perp$. 
\end{proposition}
\begin{proof}
It follows from \eqref{property:Psi:a} that $\psibf^{k\ell}\in (T_\phibf\,\Stiefel)_a^\perp$. Further, \eqref{property:Psi:b} implies that these functions are linearly independent. Taking into account that $(T_\phibf\,\Stiefel)_a^\perp$ has dimension $N(N+1)/2$, we obtain the result.
\end{proof}

Any element $\vbf\in V$ can be uniquely decomposed as $\vbf=P_\phibf^{}(\vbf)+P_\phibf^\perp(\vbf)$, where $P_\phibf^{}$ and $P_\phibf^\perp$ denote the $a_\phibf$-orthogonal projections onto $T_\phibf\,\Stiefel$ and $(T_\phibf\,\Stiefel)_a^\perp$, respectively. The projection operator $P_\phibf$ satisfies the conditions $P_\phibf\circ P_\phibf=P_\phibf$ and
\begin{subequations}
\label{property:Pphi}
\begin{align}
  \out{P_\phibf(\vbf)}{\phibf} + \out{\phibf}{P_\phibf(\vbf)} 
  &= \Nullbf, 
  \label{property:projection1}\\
  a_\phibf(\vbf - P_\phibf(\vbf), \etabf) 
  &= 0 \quad\qquad\text{for all } \etabf\in T_\phibf\,\Stiefel.
  \label{property:projection2}
\end{align}
\end{subequations}
Note that \eqref{property:Pphi} implies that $\mathrm{range}\,P_\phibf =T_\phibf\,\Stiefel$ and $\mathrm{ker}\,P_\phibf = (T_\phibf\,\Stiefel)_a^\perp$.
For the construction of such an~operator, we use the basis functions $\psibf^{k\ell}$. 
%For the construction of the projection operator~$P_\phibf$ satisfying
%\begin{subequations}
%\label{property:Pphi}
%\begin{align}
  %\out{P_\phibf(\vbf)}{\phibf} + \out{\phibf}{P_\phibf(\vbf)} &= \Nullbf, 
  %\label{property:projection1}\\
  %a_\phibf(\vbf - P_\phibf(\vbf), \etabf) &= 0 \quad\qquad\text{for all } \etabf\in T_\phibf\,\Stiefel, 
  %\label{property:projection2}
%\end{align}
%\end{subequations}
%we use the basis functions $\psibf^{k\ell}$. 
It turns out that~for any $\vbf\in V$, $P_\phibf(\vbf)$ can be written as  
\begin{align}
P_\phibf(\vbf) 
&= \vbf - \sum_{k\le \ell} (\vbf, \phibf^{k\ell})_H\, \psibf^{k\ell} \label{eq:P_phi}\\
&= \vbf - \sum_{k=1}^N (v_k, \phi_k)_{L^2(\Omega)}\, \psibf^{kk} 
- \frac{1}{\sqrt 2} \sum_{k<\ell} \Big[ (v_k, \phi_\ell)_{L^2(\Omega)} + (v_\ell, \phi_k)_{L^2(\Omega)} \Big] \, \psibf^{k\ell}.
\nonumber
\end{align}
The following result shows that this operator indeed satisfies the requested conditions
%~\eqref{property:Pphi} 
and, hence, equals the $a_\phibf$-orthogonal projection onto $T_\phibf\,\Stiefel$.
\begin{proposition}
For~$\phibf \in \Stiefel$, the operator $P_\phibf$ from~\eqref{eq:P_phi} is the $a_\phibf$-orthogonal projection onto $T_\phibf\,\Stiefel$. 
\end{proposition}
\begin{proof}
First, we emphasize that $P_\phibf$ in \eqref{eq:P_phi} is a projection, since by Proposition~\ref{prop:Phihat} all summands $(\vbf, \phibf^{k\ell})_H$ vanish if $\vbf$ is already an element of $T_\phibf\,\Stiefel$. 

Next, we verify condition~\eqref{property:projection1}, which means that $P_\phibf$ maps $V$ into $T_\phibf\,\Stiefel$. Note that for $k<\ell$, we obtain from \eqref{property:Psi:b} that
\[
	(\psi^{k\ell}_i, \phi_i^{})_{L^2(\Omega)}
	= (\psibf^{k\ell}, \phibf^{ii})_H
	= \delta_{ik}\delta_{i\ell} 
	= 0, \qquad i=1,\ldots,N.
\]
This implies $(\out{P_\phibf(\vbf)}{\phibf} + \out{\phibf}{P_\phibf(\vbf)} )_{i,i} = 2\, ( \out{P_\phibf(\vbf)}{\phibf} )_{i,i} = 0$ for all $\vbf\in V$. Further, for $i\neq j$, we observe that 
\begin{align*}
	&\big(\out{P_\phibf(\vbf)}{\phibf} + \out{\phibf}{P_\phibf(\vbf)} \big)_{i,j}\\
	&\qquad= \bigl((P_\phibf(\vbf))_i, \phi_j\bigr)_{L^2(\Omega)} + \bigl((P_\phibf(\vbf))_j, \phi_i\bigr)_{L^2(\Omega)} \\
	&\qquad= (v_i, \phi_j)_{L^2(\Omega)} + (v_j, \phi_i)_{L^2(\Omega)} \\
	&\qquad\qquad- \frac{1}{\sqrt 2} \sum_{k<\ell} \Big[ (v_k, \phi_\ell)_{L^2(\Omega)} + (v_\ell, \phi_k)_{L^2(\Omega)} \Big]\, \Big[ (\psi^{k\ell}_i, \phi_j)_{L^2(\Omega)} + (\psi^{k\ell}_j, \phi_i)_{L^2(\Omega)}\Big] \\
	&\qquad= (v_i, \phi_j)_{L^2(\Omega)} + (v_j, \phi_i)_{L^2(\Omega)} 
	- \sum_{k<\ell} \Big[ (v_k, \phi_\ell)_{L^2(\Omega)} + (v_\ell, \phi_k)_{L^2(\Omega)} \Big]\, 
	\delta_{i k}\delta_{j \ell}
	= 0. 
\end{align*}
Finally, we show the $a_\phibf$-orthogonality property \eqref{property:projection2}. 
Indeed, for any $\etabf\in T_\phibf\,\Stiefel$, \eqref{eq:P_phi} and \eqref{property:Psi:a} yield 
\[
	a_\phibf(\vbf - P_\phibf(\vbf), \etabf) 
	= \sum_{k\le \ell} (\vbf, \phibf^{k\ell})_H\, a_\phibf(\psibf^{k\ell}, \etabf)
	= 0. 
\]
Thus, $P_\phibf$ is the $a_\phibf$-orthogonal projection onto $T_\phibf\,\Stiefel$.
\end{proof}
For the Riemannian gradient descent method, which will be introduced in Section~\ref{sect:Riemannian}, we are especially interested in the projection operator~$P_\phibf$ applied to $\phibf\in \Stiefel$. In this case, we get 
\begin{equation*}
	P_\phibf(\phibf) 
	= \phibf - \sum_{k=1}^N (\phi_k, \phi_k)_{L^2(\Omega)}\, \psibf^{kk}
	= \phibf - \sum_{k=1}^N \psibf^{kk}.
\end{equation*}
Hence, for the computation of $P_\phibf(\phibf)$, one only needs the sum $\psibf := \sum_{k=1}^N \psibf^{kk}$ of the functions $\psibf^{kk} \in V$, $k=1,\dots,N$. It follows from~\eqref{property:Psi} that this sum is uniquely defined by the equations
\begin{subequations}
\label{eq:SumOfPsi}
\begin{align}
  a_\phibf(\psibf, \etabf) 
  &= 0 \hspace{2.6em}\text{ for all } \etabf\in T_\phibf\,\Stiefel, \label{eq:SumOfPsi:a} \\
  (\psibf, \phibf S^{ij})_H 
  &= \delta_{ij} \qquad\text{ for } 1\le i\le j\le N. \label{eq:SumOfPsi:b}
\end{align}
\end{subequations}
The following proposition provides an~explicit expression for the solution $\psibf$. 
\begin{proposition}\label{prop:psi}
Let~$\phibf\in \Stiefel$. The unique solution of system \eqref{eq:SumOfPsi} is given by
\begin{equation}\label{eq:psi}
\psibf  = \mathcal A_\phibf^{-1}\phibf\, \out{\phibf}{\mathcal A_\phibf^{-1}\phibf}^{-1}.
\end{equation}
\end{proposition}
\begin{proof}
System \eqref{eq:SumOfPsi} is equivalent to the saddle point problem 
\begin{subequations}
\label{eq:SumOfPsi_spp}
\begin{align}
  a_\phibf(\psibf, \vbf) + \sum_{i\leq j} (\phibf S^{ij},\vbf)_H \mu^{ij} &= 0 \hspace{2.6em}\text{ for all } \vbf\in V, \label{eq:SumOfPsi_spp:a} \\
  (\psibf, \phibf S^{ij})_H \hspace*{29mm}
  &=  \delta_{ij} \qquad\text{ for } 1\le i\le j\le N \label{eq:SumOfPsi_spp:b}
\end{align}
\end{subequations}
for $\psibf\in V$ and the Lagrange multipliers $\mu^{ij}\in\mathbb{R}$. Using the special structure of the matrices $S^{ij}$ in \eqref{eq:Sij}, the constraint conditions \eqref{eq:SumOfPsi_spp:b} can be written as
\mbox{$\mathrm{sym}\bigl( \out{\psibf}{\phibf}\bigr)=\Onebf$}, where $\mathrm{sym}(A)=\tfrac{1}{2}(A+A^T)$ denotes the symmetric part of a~matrix~$A\in\mathbb{R}^{N\times N}$. Further, we obtain
$$
\sum_{i\leq j} (\phibf S^{ij},\vbf)_H \mu^{ij} = \bigl(\phibf \sum_{i\leq j} S^{ij}\mu^{ij},\vbf \bigr)_H=(\phibf\,S,\vbf)_H
$$ 
with the symmetric matrix $S=\sum_{i\leq j} S^{ij}\mu^{ij}$. As a result, system \eqref{eq:SumOfPsi_spp}
takes the form
\begin{subequations}
\label{eq:SumOfPsi_spp2}
\begin{align}
  a_\phibf(\psibf, \vbf) + (\phibf\,S,\vbf)_H &= 0 \hspace{2.6em}\text{ for all } \vbf\in V, \label{eq:SumOfPsi_spp2:a} \\
  \mathrm{sym}\bigl( \out{\psibf}{\phibf} \bigr)\hspace*{12mm}
  &= \Onebf. \label{eq:SumOfPsi_spp2:b}
\end{align}
\end{subequations}
Using \eqref{eq:invA}, we derive from \eqref{eq:SumOfPsi_spp2:a} that
$$
0=a_\phibf(\psibf, \vbf) +a_\phibf(\calA_\phibf^{-1}\phibf\,S, \vbf) = a_\phibf(\psibf+\calA_\phibf^{-1}\phibf\,S, \vbf)
\qquad \text{ for all } \vbf\in V
$$
and, hence, $\psibf=-\calA_\phibf^{-1}\phibf\,S$. Substituting this function into \eqref{eq:SumOfPsi_spp2:b} yields the Lyapunov equation
\begin{equation}\label{eq:Lyapunov}
  \out{\phibf}{\mathcal A_\phibf^{-1}\phibf} S + S\, \out{\phibf}{\mathcal A_\phibf^{-1}\phibf} 
  = -2\Onebf
\end{equation}
for $S$. By Proposition~\ref{prop:symmetr}, the matrix $\out{\phibf}{\mathcal A_\phibf^{-1}\phibf}$ is symmetric positive definite. In this case, the Lyapunov equation \eqref{eq:Lyapunov} has a~unique symmetric solution \cite[Th.~12.3.2]{LancT85} given by 
$S = -\out{\phibf}{\mathcal A_\phibf^{-1}\phibf}^{-1}$. This finally gives the expression \eqref{eq:psi}.
\end{proof}
%
%
%=============================================================================
%
\subsection{Retractions}\label{sec:retraction}
Next, we introduce the concept of retractions on the Stiefel manifold $\Stiefel$. Retractions provide a useful tool in Riemannian optimization which allows us to keep the iteration points on the manifold. 
\begin{definition}[Retraction]\label{def:retraction}
Let $T\,\Stiefel$ be the tangent bundle of $\Stiefel$. A~smooth map \mbox{$\calR\colon T\,\Stiefel\to\Stiefel$} is called a~{\em retraction on $\Stiefel$} if for all $\phibf\in\Stiefel$, the restriction 
$\calR_\phibf=\calR\bigl|_{T_\phibf\,\Stiefel}\bigr.$ on $T_\phibf\,\Stiefel$ has the following properties:
\begin{itemize}
\setlength\itemsep{0.5em}
\item[a)] $\calR_\phibf(\bm{0}_\phibf) = \calR(\phibf, \bm{0}_\phibf) = \phibf$, where $\bm{0}_\phibf$ denotes the zero element of $T_\phibf\,\Stiefel$, 
\item[b)] $\tfrac{{\rm d}}{{\rm d}t} \calR_\phibf(t\etabf)\bigl|_{t=0}\bigr. = \etabf$ for all $\etabf\in T_\phibf\,\Stiefel$.
\end{itemize}
\end{definition}
In \cite[Ex.~4.1.3]{AbsiMS08} and \cite{AbsiM12,EdeAS98,KanFT13,SatA19}, several retractions on the (generalized) Stiefel matrix manifold have been introduced and compared with respect to computational cost and accuracy. Here, we extend some of the decomposition-based retractions to the manifold~$\Stiefel$.
%
%=======================================================================
\subsubsection{The projective retraction}\label{sec:retraction:polar}
First, we introduce a~retraction based on the polar decomposition and show that it provides a~projection onto~$\Stiefel$. 

Similarly to the matrix case, e.g., \cite[Sect.~9.4.3]{GoluV13}, we define the {\em polar decomposition} of $\vbf\in V$ as
$\vbf=\ubf S$, where $\ubf\in\Stiefel$ and $S\in\mathbb{R}^{N\times N}$ is symmetric positive semidefinite. Such a~decomposition
always exists. If the components of $\vbf$ are linearly independent, then the matrix $\out{\vbf}{\vbf}$ is positive definite. In this case, $S=\out{\vbf}{\vbf}^{1/2}$ is positive definite and the factor $\ubf=\vbf\, \out{\vbf}{\vbf}^{-1/2}$ is unique. 

For any $(\phibf,\etabf)\in T\,\Stiefel$, i.e., $\etabf\in T_\phibf\,\Stiefel$, the components of $\phibf+\etabf$ are linearly independent, since the matrix
\begin{equation}\label{eq:phi_eta}
  \out{\phibf+\etabf}{\phibf+\etabf} 
  = \out{\phibf}{\phibf} + \out{\phibf}{\etabf} + \out{\etabf}{\phibf} + \out{\etabf}{\etabf}
 = \Onebf + \out{\etabf}{\etabf}
\end{equation}
is positive definite. Then we can use the polar decomposition of $\phibf+\etabf$ to define a~retraction on $\Stiefel$.
\begin{proposition}
For $(\phibf,\etabf)\in T\,\Stiefel$, the map 
\begin{equation}\label{eq:retr_pol}
  \calR(\phibf,\etabf)
  :=(\phibf+\etabf)\bigl(\Onebf+\out{\etabf}{\etabf}\bigr)^{-1/2} 
\end{equation}
 is a~retraction on $\Stiefel$. 
\end{proposition}
\begin{proof}
Let $(\phibf,\etabf)\in T\,\Stiefel$. First, we verify that $\calR(\phibf,\etabf)$ belongs to $\Stiefel$. 
%Using the definitions of~$\Stiefel$ and $T_\phibf\,\Stiefel$ as well as 
Using Lemma~\ref{lem:propOuterProduct} and \eqref{eq:phi_eta}, we obtain
\begin{align*}
  \out{\calR(\phibf,\etabf)}{\calR(\phibf,\etabf)}
  &= \big\llbracket (\phibf+\etabf)\bigl(\Onebf+\out{\etabf}{\etabf}\bigr)^{-1/2},(\phibf+\etabf)\bigl(\Onebf+\out{\etabf}{\etabf}\bigr)^{-1/2} \big\rrbracket_H \\
  %&= \bigl(\Onebf+\out{\etabf}{\etabf}\bigr)^{-1/2}\bigl(\out{\phibf}{\phibf}+\out{\phibf}{\etabf}+\out{\etabf}{\phibf}+\out{\etabf}{\etabf}\bigr) \bigl(\Onebf+\out{\etabf}{\etabf}\bigr)^{-1/2}\\
  &= \bigl(\Onebf+\out{\etabf}{\etabf}\bigr)^{-1/2}\bigl(\Onebf+\out{\etabf}{\etabf}\bigr)\bigl(\Onebf+\out{\etabf}{\etabf}\bigr)^{-1/2} = \Onebf,
\end{align*}
and, hence, $\calR(\phibf,\etabf)\in\Stiefel$. Furthermore, we have $\calR_\phibf(\bm{0}_\phibf)=\phibf$ and 
\begin{align*}
\tfrac{{\rm d}}{{\rm d}t} \calR_\phibf(t\etabf)\Bigl|_{t=0}\Bigr. 
&= 
\tfrac{{\rm d}}{{\rm d}t} (\phibf+t\etabf)\bigl(\Onebf+t^2\out{\etabf}{\etabf}\bigr)^{-1/2}\Bigl|_{t=0}\Bigr.\\
&= -t\,(\phibf+t\etabf)\, \out{\etabf}{\etabf}\bigl(\Onebf+t^2\out{\etabf}{\etabf}\bigr)^{-3/2}+\etabf\,\bigl(\Onebf+t^2\out{\etabf}{\etabf}\bigr)^{-1/2}\Bigl|_{t=0}\Bigr.\\
&=\etabf.
\end{align*}
This shows that $\calR$ defined in~\eqref{eq:retr_pol} is the retraction on $\Stiefel$. 
\end{proof}
The evaluation of the retraction in \eqref{eq:retr_pol} involves the computation of the outer pro\-duct~$\out{\etabf}{\etabf}$ and the eigenvalue decomposition  \begin{equation}\label{eq:VDVt}
  \Onebf + \out{\etabf}{\etabf} 
  = QDQ^T,
\end{equation}
where $Q\in\mathbb{R}^{N\times N}$ is orthogonal and $D=\diag(d_1,\ldots,d_N)$ with $d_j>0$ for $j=1,\ldots,N$. With this, we obtain $\calR(\phibf,\etabf)=(\phibf+\etabf)\,QD^{-1/2}Q^T$.
\begin{remark}
For stability reasons, we recommend to use $\out{\phibf+\etabf}{\phibf+\etabf}$ instead of $\Onebf+\out{\etabf}{\etabf}$ in \eqref{eq:VDVt}. 
A~similar suggestion for the generalized Stiefel matrix manifold can be found in~\cite{SatA19}. 
Note that, due to~\eqref{eq:phi_eta}, both expressions are equivalent if $\phibf\in\Stiefel$ and \mbox{$\etabf\in T_\phibf\,\Stiefel$}.
\end{remark}
The polar decomposition based retraction \eqref{eq:retr_pol} can be viewed as a~projective retraction, since it satisfies
\begin{equation}\label{eq:ProjRetr}
\calR(\phibf,\etabf)=\argmin_{\xibf\in\Stiefel}\|\xibf-(\phibf+\etabf)\|_H^2.
\end{equation}
To prove this, we first observe that for all $(\phibf,\etabf)\in T\,\Stiefel$, $\phibf+\etabf$ can be represented~as 
\begin{equation}\label{eq:svd}
\phibf+\etabf=\ubf D^{1/2} Q^T,
\end{equation}
where $\ubf\in\Stiefel$ and $D,Q$ are as in \eqref{eq:VDVt}. This decomposition is an~extension of the singular value decomposition known for matrices, e.g., \cite[Sect.~2.4]{GoluV13} to the elements of $V$. For any \mbox{$\xibf\in\Stiefel$}, we have 
\begin{align*}
\|\xibf-(\phibf+\etabf)\|_H^2 & = \|\xibf\|_H^2-2\,(\xibf,\phibf+\etabf)_H+\|\phibf+\etabf\|_H^2  
= N^2-2\,(\xibf,\phibf+\etabf)_H+\trace D
\end{align*}
with
\begin{align*}
  (\xibf,\phibf+\etabf)_H 
  &= \trace\, \big( \out{\xibf}{\ubf\, D^{1/2}Q^T} \big)
  = \trace\, \big( \out{\xibf}{\ubf} D^{1/2} \big)\\
  &= \sum_{i=1}^N (\xi_i,u_i)_{L^2(\Omega)}\sqrt{d_i}\leq  \sum_{i=1}^N \|\xi_i\|_{L^2(\Omega)}\|u_i\|_{L^2(\Omega)}\sqrt{d_i} =\trace D^{1/2}.
\end{align*}
For $\xibf=\ubf\,Q^T\in\Stiefel$, the equality 
$$
(\xibf,\phibf+\etabf)_H = \trace\, \out{\ubf\,Q^T}{\ubf\, D^{1/2}Q^T}=\trace D^{1/2} 
$$
holds, i.e., $\xibf=\ubf\,Q^T$ solves \eqref{eq:ProjRetr}. Thus, \mbox{$\calR(\phibf,\etabf)=(\phibf+\etabf)QD^{-1/2}Q^T = \ubf\,Q^T$} is a~projection onto $\Stiefel$.

The following proposition shows that the retraction \eqref{eq:retr_pol} is second-order bounded. 
\begin{proposition}\label{prop:retr_est}
The retraction $\calR$ in \eqref{eq:retr_pol} satisfies
\begin{equation*}%\label{eq:retr_est}
	\|\calR(\phibf, t\etabf)-(\phibf+t\etabf)\|_{a_\phibf}\leq 
	t^2\,\|\phibf+ t\etabf\|_{a_\phibf}\|\etabf\|_H^2.
\end{equation*}
\end{proposition}
\begin{proof}
The proof is given in Appendix~\ref{app:proof:projective}. 
\end{proof}	
%
%====================================================================
%
\subsubsection{The $qR$-based retraction}\label{sec:retraction:qR}
An alternative retraction on $\Stiefel$ can be defined by using the orthonormalization with respect to the inner product~$(\,\cdot\,, \cdot\,)_H$. First, we observe that for any $\vbf=(v_1,\ldots,v_N)\in V$ with linearly independent components, there exist \mbox{$\qbf\in\Stiefel$} and an~upper triangular matrix $R\in\mathbb{R}^{N\times N}$ with strictly positive diagonal ele\-ments such that $\vbf=\qbf R$. The existence of such a~decomposition, called {\em $qR$ decomposition}, can be proved constructively by using the Gram-Schmidt orthonormalization procedure
\begin{equation}\label{eq:GramSchmidt}
\begin{array}{ll}
\tilde{q}_1 := v_1, &\displaystyle{\qquad q_1:=\frac{\tilde{q}_1}{\|\tilde{q}_1\|_{L^2(\Omega)}}, }\\
\tilde{q}_j := v_j - \displaystyle{\sum\limits_{i=1}^{j-1}} (v_j, q_i)_{L^2(\Omega)}\, q_i, &\qquad
\displaystyle{q_j := \frac{\tilde{q}_j}{\|\tilde{q}_j\|_{L^2(\Omega)}}, \qquad j=2,\ldots, N.}
\end{array}
\end{equation}
With this, we obtain $\qbf=(q_1,\ldots,q_N)\in\Stiefel$ and 
$$
R=\begin{bmatrix} 
(v_1,q_1)_{L^2(\Omega)} & (v_2,q_1)_{L^2(\Omega)} & \cdots & \cdots & (v_N,q_1)_{L^2(\Omega)} \\
0                 & (v_2,q_2)_{L^2(\Omega)} & \cdots & \cdots & (v_N,q_2)_{L^2(\Omega)} \\
0                &	0                  & \ddots & & \vdots \\
\vdots             &	                 & \ddots & \ddots & \vdots \\
0                 &  0      & \cdots & 0  & (v_N,q_N)_{L^2(\Omega)}
\end{bmatrix}.
$$
Note that the matrix $R$ has positive diagonal elements $(v_j,q_j)_{L^2(\Omega)}=\|\tilde{q}_j\|_{L^2(\Omega)}$. This property of $R$ guarantees the uniqueness of the $qR$ decomposition. Let $\qf(\vbf)$ denote the factor $\qbf$ in $\vbf=\qbf R$. This allows us to define a~$qR$-based retraction on the Stiefel manifold~$\Stiefel$. 
\begin{proposition}
For~$(\phibf,\etabf)\in T\,\Stiefel$, the map 
\begin{equation}\label{eq:retr_qR}
\calR(\phibf,\etabf):=\qf(\phibf+\etabf)
\end{equation}
 is a~retraction on $\Stiefel$. 
\end{proposition}
\begin{proof}
Obviously, $\calR$ in \eqref{eq:retr_qR} is well defined on $T\,\Stiefel$. Further, by definition, we have $\calR(\phibf,\etabf)\in\Stiefel$ for all $(\phibf,\etabf)\in T\,\Stiefel$ and $\calR_\phibf(\bm{0}_\phibf)=\phibf$. 

In order to prove the second property in Definition~\ref{def:retraction}, we follow the lines of~\cite[Ex.~8.1.5]{AbsiMS08}.
For any $(\phibf,\etabf)\in T\,\Stiefel$, we consider a~curve $\varphibf(t)=\phibf+t\etabf$. Let \mbox{$\varphibf(t)=\qbf(t)R(t)$} be the $qR$~decomposition of $\varphibf(t)$. Then, using the product rule, we have
\begin{equation}\label{eq:dotPsi}
\dot{\varphibf}(t)=\dot{\qbf}(t)R(t)+\qbf(t)\dot{R}(t),
\end{equation}
where $\dot{\varphibf}(t)=\tfrac{\rm d}{{\rm d}t}\varphibf(t)$ and similar for $\qbf(t)$ and $R(t)$. For the sake of brevity, we omit the argument $t$ in what follows. Computing the outer product of $\qbf$ and $\dot{\varphibf}$, we obtain %by using \eqref{eq:dotPsi} that 
\begin{equation}\label{eq:qDotPsi}
  \out{\qbf}{\dot{\varphibf}} 
  = \out{\qbf}{\dot{\qbf}} R + \out{\qbf}{\qbf} \dot{R}
  = \out{\qbf}{\dot{\qbf}} R + \dot{R}.
\end{equation}
Multiplication of~\eqref{eq:qDotPsi} by $R^{-1}$ from the right yields 
$$
  \out{\qbf}{\dot{\varphibf}} R^{-1} 
  = \out{\qbf}{\dot{\qbf}} + \dot{R}R^{-1}, 
  %=: \varrho_{\rm skew}(\out{\qbf}{\dot{\psibf}} R^{-1})+\varrho_{\rm up}(\out{\qbf}{\dot{\psibf}} R^{-1}),
$$
where $\out{\qbf}{\dot{\qbf}}$ is skew-symmetric and~$\dot{R}R^{-1}$ is upper triangular. Since~$M:=\out{\qbf}{\dot{\varphibf}} R^{-1}$ can uniquely be represented as $M=\varrho_{\rm skew}(M)+\varrho_{\rm up}(M)$, where 
$\varrho_{\rm skew}(M)$ is skew-sym\-met\-ric and $\varrho_{\rm up}(M)$ is upper triangular, we obtain 
\[
  \varrho_{\rm skew}(\out{\qbf}{\dot{\varphibf}} R^{-1})
  = \out{\qbf}{\dot{\qbf}}, \qquad 
  \varrho_{\rm up}(\out{\qbf}{\dot{\varphibf}} R^{-1}) 
  = \dot{R}R^{-1}.
\]
Further, multiplying~\eqref{eq:qDotPsi} by~$\qbf$ from the left and subtracting the resulting equation from \eqref{eq:dotPsi},
we find 
$$
\dot{\varphibf}-\qbf\, \out{\qbf}{\dot{\varphibf}} 
= \dot{\qbf}R-\qbf\, \varrho_{\rm skew}(\out{\qbf}{\dot{\varphibf}} R^{-1})R,
$$
which implies 
$$
\dot{\qbf} 
= (\dot{\varphibf}-\qbf\, \out{\qbf}{\dot{\varphibf}})R^{-1} + \qbf\, \varrho_{\rm skew}(\out{\qbf}{\dot{\varphibf}} R^{-1}).
$$
Taking into account that $\dot{\varphibf}(0)=\etabf$, $\varphibf(0)=\phibf=\qbf(0)$, $R(0)=\Onebf$, and that~$\out{\phibf}{\etabf}$ is skew-symmetric, we finally obtain
$$
\tfrac{{\rm d}}{{\rm d}t} \calR_\phibf(t\etabf)\Bigl|_{t=0}\Bigr. 
= \tfrac{{\rm d}}{{\rm d}t}\, \qbf(t)\Bigl|_{t=0}\Bigr.
= \etabf-\phibf\,\out{\phibf}{\etabf}+\phibf\,\out{\phibf}{\etabf}= \etabf, 
$$
which completes the proof.
\end{proof}
The $qR$-based retraction \eqref{eq:retr_qR} can be computed by the modified Gram-Schmidt procedure as presented in Algorithm~\ref{alg:modGS} which is more numerically stable than the Gram-Schmidt process~\eqref{eq:GramSchmidt}. An~alternative approach for evaluating \eqref{eq:retr_qR} is based on computing the Cholesky factorization $\out{\phibf+\etabf}{\phibf+\etabf}=F^TF$ with an~upper triangular matrix $F\in\mathbb{R}^{N\times N}$ and determining 
\begin{equation} \label{eq:retrChol}
\calR(\phibf,\etabf)=(\phibf+\etabf)\, F^{-1}.
\end{equation}
It is an~extension of the Cholesky-QR-based method on the generalized matrix Stiefel ma\-ni\-fold presented in \cite{SatA19}. Compared to the polar decomposition based retraction~\eqref{eq:retr_pol}, the computation of \eqref{eq:retrChol} has lower numerical complexity, especially for large $N$, since it requires the Cholesky factorization instead of the eigenvalue decomposition.
\begin{algorithm}[t]
\setstretch{1.15}
\caption{Modified Gram-Schmidt procedure}
\label{alg:modGS}
\begin{algorithmic}[1]
	\State {\bf Input}: $\vbf = (v_1, \dots, v_N) \in V$ 
	%\vspace{0.5em}
	%
	\For{$i=1,\ldots,N$} 
      \State{$r_{ii}=\|v_i\|_{L^2(\Omega)}$}
			\State{$q_i=v_i/r_{ii}$}
      \For{$j=i+1,\ldots,N$}
			    \State{$r_{ij}=(v_j,q_i)_{L^2(\Omega)}$}
					\State{$v_j=v_j-r_{ij}q_i$}
      \EndFor
	\EndFor
	%\vspace{0.5em}	
	%				
	\State {\bf Output}: $\qbf=(q_1,\ldots,q_N)\in \Stiefel$ and $R=[r_{ij}]\in\mathbb{R}^{N\times N}$ such that $\vbf=\qbf R$
\end{algorithmic}
\end{algorithm} 

The following proposition establishes the second-order boundedness of the $qR$-based retraction \eqref{eq:retr_qR}.
\begin{proposition}\label{prop:retr_est_qR}
The retraction $\calR$ in~\eqref{eq:retr_qR} satisfies
\[
	\|\calR(\phibf, t\etabf)-(\phibf+t\etabf)\|_{a_\phibf}\leq 
	\frac{t^2}{\sqrt{2}}\,\|\phibf+ t\etabf\|_{a_\phibf}\bigl(1+t^2\|\etabf\|_H^2\bigr)^{1/2}\|\etabf\|_H^2.
\]
\end{proposition}
\begin{proof}
The proof is given in Appendix~\ref{app:proof:qR}. 
\end{proof}
%
%
%=============================================================================
%=========  Riemannian gradient method
%=============================================================================
%
\section{Energy-Adaptive Riemannian Gradient Descent Method}\label{sect:Riemannian}
The simplest approach to minimize the energy functional $\calE$ over $\Stiefel$ is the gradient descent method, which requires the Riemannian gradient of~$\calE$. For a~smooth scalar field~$\calE$ on the Riemannian manifold $\Stiefel$, the~{\em Riemannian gradient} $\grad\calE(\phibf)$ of~$\calE$ at~\mbox{$\phibf\in\Stiefel$} with respect to the metric $g_a$ is defined as the unique element of the tangent space $T_\phibf\,\Stiefel$ satisfying 
$$
g_a(\grad\calE(\phibf),\etabf)
= a_{\phibf}(\grad\calE(\phibf),\etabf) 
= \mathrm{D}\calE(\phibf)[\etabf]
\qquad \text{ for all } \etabf\in T_\phibf\,\Stiefel.
$$ 
Since $\Stiefel$ is an embedded submanifold of $V$, we obtain the following expression for the Riemannian gradient.
\begin{proposition}
The Riemannian gradient of~the energy functional~$\calE\colon V\to\mathbb{R}$ from~\eqref{eq:energy} at $\phibf\in\Stiefel$ with respect to the metric $g_a$ is given by 
\begin{equation}\label{eq:gradE}
\grad\calE(\phibf) 
= P_\phibf(\phibf)
= \phibf - \mathcal A_\phibf^{-1}\phibf\, \out{\phibf}{\mathcal A_\phibf^{-1}\phibf}^{-1}.
\end{equation}
\end{proposition}
\begin{proof}
Using~\eqref{eq:DerivativeE}, we obtain 
$$
a_{\phibf}(\grad\calE(\phibf),\etabf) = \mathrm{D}\calE(\phibf)[\etabf] = a_{\phibf}(\phibf,\etabf)\qquad  \text{ for all } \etabf\in T_\phibf\,\Stiefel.
$$
Hence, $a_{\phibf}(\grad\calE(\phibf)-\phibf,\etabf)=0$ for all $\etabf\in T_\phibf\,\Stiefel$. This implies that $\grad\calE(\phibf)-\phibf$ belongs to the normal space $(T_\phibf\,\Stiefel)_a^\perp$ and, hence, 
$$
\grad\calE(\phibf)
= \phibf + P_\phibf^\perp\bigl(\grad\calE(\phibf)-\phibf\bigr) 
= \phibf - P_\phibf^\perp(\phibf)
= P_\phibf(\phibf).
$$
The second expression for $\grad\calE(\phibf)$
in~\eqref{eq:gradE} immediately follows from Proposition~\ref{prop:psi}.% from~\eqref{eq:PPhi}.
\end{proof}
Using the Riemannian gradient and any retraction $\calR$ on $\Stiefel$ from Section~\ref{sec:retraction}, the Riemannian gradient descent method for solving the minimization problem \eqref{eq:opt} can be formulated as follows: for given~$\phibf^{(n)}\in\Stiefel$, compute 
\begin{equation}\label{eq:RiemannianGradMethod}
  \phibf^{(n+1)} = \calR(\phibf^{(n)},\tau_n \etabf^{(n)})
\end{equation}
with the~search direction $\etabf^{(n)}=-\grad\calE(\phibf^{(n)})$ and an~appropriately chosen step size \mbox{$\tau_n>0$}. 
%Due to~\eqref{eq:gradE}, this results in the iteration~\eqref{eq:theMethod}. 

%The presented minimization approach for solving the nonlinear eigenvector problem~\eqref{eq:EVPweakComponents} is closely related to the Sobolev gradient flow algorithm studied in~\cite{HenP20} for the Gross-Pitaevskii eigenvalue problem. As we will show in Section~\ref{sect:examples}, the Gross-Pitaevskii eigenvalue problem fits in the given framework with~$N=1$. 

\begin{remark}[Connection to Sobolev gradient flows]
The presented minimization approach for solving the nonlinear eigenvector problem~\eqref{eq:EVPweakComponents} is closely related to the Sobolev gradient flow algorithm studied in~\cite{HenP20} for the Gross-Pitaevskii eigenvalue problem which, as will be shown in Section~\ref{sect:examples}, fits in the given framework with~$N=1$. For general problems with $N\geq 1$, let $\nabla\calE(\phibf)$ denote the~Riesz representative of $\mathrm{D}\calE(\phibf)$ in the Hilbert space $V$ with respect to the inner product $a_\phibf(\,\cdot\,,\cdot\,)$. The operator $\nabla\calE\colon V\to V$ is called the {\em $a_\phibf$-Sobolev gradient of~$\calE$}. It follows from \eqref{eq:DerivativeE} that
$$
	a_\phibf(\nabla\calE(\phibf),\vbf)=\mathrm{D}\calE(\phibf)[\vbf]
	= a_\phibf(\phibf,\vbf) \qquad\text{for all } \vbf \in V
$$
and, hence, $\nabla\calE(\phibf)=\phibf$. 
Given an~initial guess $\phibf(0) \in \Stiefel$, the corresponding dynamical system, also called the $a_\phibf$-Sobolev gradient flow, has the form 
\begin{align}
\label{eq:gradFlow}
	\dot \phibf(t)
	= - P_{\phibf(t)}(\nabla\calE(\phibf(t)))
	= - P_{\phibf(t)}(\phibf(t))
	= - \grad\calE(\phibf(t)).
\end{align}
It can be easily seen that the solution of this system satisfies~$\phibf(t) \in \Stiefel$ for all times. 
Moreover, any stationary solution $\phibf^*\in\Stiefel$ of~\eqref{eq:gradFlow} is the critical point of the energy~$\calE$ in \eqref{eq:energy}, since it satisfies $\grad\calE(\phibf^*)=0$.
\end{remark}

In the following subsection, we show that the iteration \eqref{eq:RiemannianGradMethod} is convergent if the step size $\tau_n$ is sufficiently small. 
%
%
%=========================================================================
\subsection{Convergence analysis}
To show that the Riemannian gradient scheme~\eqref{eq:RiemannianGradMethod} converges, we restrict ourselves to the case of a~constant step size $\tau_n\equiv\tau$. First, we collect some assumptions which guarantee the convergence as established in Theorem~\ref{th:convergence} below. %Afterwards, we prove that these assumptions are fulfilled in our setting. 
\begin{enumerate}[label={\textbf{(A\arabic*)}}]
\item\label{A1} (Polyak-{\L}ojasiewicz gradient inequality)
For the ground state~$\phibf^*\in\Stiefel$, there exist~$C_*, C_\text{PL}>0$ such that for all 
$\phibf\in\Stiefel$ with $\|\phibf-\phibf^*\|_{a_0}\leq C_*$, it holds 
$$
  \bigl|\calE(\phibf) - \calE(\phibf^{*})\bigr|
  \le C_\text{PL}\, \|\grad \calE(\phibf)\|_{a_{\phibf}}^2. 
$$
\item\label{A2} (Descent inequality)
%For a~sequence $\{\phibf^{(n)}\}\subset \Stiefel$, 
We say that a given sequence $\{\phibf^{(n)}\}\subset \Stiefel$ satisfies the descent inequality, if there exist $C_\text{D}>0$ and $n_D\in\mathbb{N}$ such that for all $n\geq n_D$,
\begin{equation}\label{eq:energy_decr}
\calE(\phibf^{(n)}) - \calE(\phibf^{(n+1)}) \geq C_\text{D}\, \|\grad\,\calE(\phibf^{(n)})\|_{a_{\phibf^{(n)}}}\|\phibf^{(n+1)}-\phibf^{(n)}\|_{a_0}.
\end{equation}
\item\label{A3} (Step size condition)
For a given sequence $\{\phibf^{(n)}\}\subset \Stiefel$, we say that it satisfies the step size condition, if there exist $C_\text{S}>0$ and $n_S\in\mathbb{N}$ such that for all $n\geq n_S$,
\begin{equation}\label{eq:step_size}
\|\phibf^{(n+1)}-\phibf^{(n)}\|_{a_0} \geq C_\text{S}\, \|\grad\,\calE(\phibf^{(n)})\|_{a_{\phibf^{(n)}}}.
\end{equation}
\end{enumerate}
Under these assumptions, the convergence result and the convergence rate can be established by the following theorem adapted from~\cite{Zha21ppt}. Its proof is a straight-forward modification of \cite[Th.~2.1]{Zha21ppt} and therefore omitted here.
\begin{theorem}\label{th:convergence} 
Let $\{\phibf^{(n)}\}\subset \Stiefel$ be a~sequence generated by the descent gradient method \eqref{eq:RiemannianGradMethod}, which satisfies the descent condition~\textup{\ref{A2}}. 
If there exists a cluster point $\phibf^*\in\Stiefel$ of the sequence that satisfies the Polyak-{\L}ojasiewicz gradient condition~\textup{\ref{A1}}, then $\phibf^*$ is the unique limit point of $\{\phibf^{(n)}\}$ with respect to $\|\cdot\|_{a_0}$. Further, if the sequence~$\{\phibf^{(n)}\}$ fulfills the step size condition~\textup{\ref{A3}}, then there exist constants $c, C>0$ such that the convergence rate can be estimated as 
$$
  \|\phibf^{(n)}-\phibf^*\|_{a_0} 
  \leq C\, e^{-cn}
$$
and it holds $\lim\limits_{n\to\infty} \grad\,\calE(\phibf^{(n)})=0$.
\end{theorem}
It remains to discuss the validity of the three conditions~\ref{A1}--\ref{A3} in the considered setting. Condition~\ref{A1} is an assumption on the energy and depends on the particular application. The special case of the Gross-Pitaevskii equation is discussed in detail in~\cite{Zha21ppt}. 
The other two conditions can be verified under moderate constraints on the step size and suitable regularity assumptions on the energy.
\begin{lemma}[Sufficient condition for~\ref{A2}]\label{lem:condA2}
Consider a~sufficiently small step size \linebreak $0<\tau\le\tau_{\max}$. Assume that the second-order derivative of the energy is bounded in the sense that 
\begin{equation}\label{eq:SecDerBound}
  D^2\calE(\xibf)[\vbf, \wbf]
  \le C_0\, \|\vbf \|_{a_{0}} \|\wbf \|_{a_{0}} 
\end{equation}
%with $\|\vbf\|_{a_0}=\sqrt{a_0(\vbf,\vbf)}$	
for all $\xibf$ in a small neighborhood of the ground state and all~$\vbf, \wbf \in V$. If the iterates~$\phibf^{(n)}$ given by~\eqref{eq:RiemannianGradMethod} with the polar decomposition based retraction~\eqref{eq:retr_pol} are in this neighborhood, then there exists a~constant $C_\text{D}>0$ such that %for large enough~$n$, we have
the estimate~\eqref{eq:energy_decr} is satisfied.
\end{lemma}
\begin{proof}	
For $\etabf^{(n)}=-\grad\calE(\phibf^{(n)})=-\phibf^{(n)}+\psibf^{(n)}$ with 
$\psibf^{(n)}\in (T_{\phibf^{(n)}}\Stiefel)_{a}^\perp$,
we obtain 
$a_{\phibf^{(n)}}(\phibf^{(n)},\etabf^{(n)})=-a_{\phibf^{(n)}}(\etabf^{(n)},\etabf^{(n)}) =-\|\etabf^{(n)}\|_{a_{\phibf^{(n)}}}^2$.
%\begin{align*}
	%a_{\phibf^{(n)}}\bigl(\etabf^{(n)},\etabf^{(n)}\bigr) & =a_{\phibf^{(n)}}\bigl(\phibf^{(n)},-\etabf^{(n)}\bigr)\\
	%& = a_{\phibf^{(n)}}\bigl(\phibf^{(n)}-\calA_{\phibf^{(n)}}^{-1}\phibf^{(n)}\bigl[\phibf^{(n)},\calA_{\phibf^{(n)}}^{-1}
	%\phibf^{(n)}\bigr]_H,\phibf^{(n)}\bigr) \\
	%& = a_{\phibf^{(n)}}\bigl(\phibf^{(n)},\phibf^{(n)}\bigr)-\bigl(\phibf^{(n)}\bigl[\phibf^{(n)},\calA_{\phibf^{(n)}}^{-1}
	%\phibf^{(n)}\bigr]_H,\phibf^{(n)}\bigr)_H \\
	%& = a_{\phibf^{(n)}}\bigl(\phibf^{(n)},\phibf^{(n)}\bigr)-\trace\bigl(\bigl[\phibf^{(n)},\calA_{\phibf^{(n)}}^{-1}\phibf^{(n)}
	%\bigr]_H\bigr).
%\end{align*}
Further, it follows from Proposition~\ref{prop:retr_est} and
\begin{equation}\label{eq:RelDifGrad}
	\phibf^{(n+1)}-\phibf^{(n)}
	= \calR\bigl(\phibf^{(n)},\tau\,\etabf^{(n)}\bigr) - \big(\phibf^{(n)} + \tau\,\etabf^{(n)} \big) + \tau\,\etabf^{(n)}
\end{equation}
that
\begin{align}
  \|\phibf^{(n+1)}-\phibf^{(n)}\|_{a_0}
	& \le \|\phibf^{(n+1)}-\phibf^{(n)}\|_{a_{\phibf^{(n)}}} \nonumber\\
  & \le \tau\, \|\etabf^{(n)}\|_{a_{\phibf^{(n)}}} 
  + \tau^2\,\|\phibf^{(n)}+ \tau\, \etabf^{(n)}\|_{a_{\phibf^{(n)}}}\|\etabf^{(n)}\|_H^2. \label{eq:est1}
\end{align}
Using the expression $\etabf^{(n)}=-\phibf^{(n)}+\calA_{\phibf^{(n)}}^{-1}\phibf^{(n)}\out{\phibf^{(n)}}{\calA_{\phibf^{(n)}}^{-1}\phibf^{(n)}}^{-1}$ and the coercivity and boundedness of the bilinear form $a_{\phibf^{(n)}}$, we can show that there exists a~constant $C_1>0$ such that $\|\etabf^{(n)}\|_{a_{\phibf^{(n)}}}\leq C_1 \|\phibf^{(n)}\|_{a_{\phibf^{(n)}}}$. Then taking into account that the iterates $\phibf^{(n)}$ are in a~small neighborhood of the ground state, we estimate
\begin{equation}\label{eq:UniformBound}
\|\phibf^{(n)}\|_{a_{\phibf^{(n)}}}\le C_2, \quad
\|\etabf^{(n)}\|_{a_{\phibf^{(n)}}}\le C_1C_2, \quad 
\|\phibf^{(n)}+ \tau\, \etabf^{(n)}\|_{a_{\phibf^{(n)}}} \le (1+\tau_{\max}C_1)\,C_2
\end{equation}
with a constant $C_2>0$ independent of $\phibf^{(n)}$.

A Taylor expansion of $\calE(\phibf^{(n+1)})$ at $\phibf^{(n)}$ yields
\begin{align*}
\calE(\phibf^{(n+1)}) 
=\calE(\phibf^{(n)}) + {\rm D}\calE(\phibf^{(n)})[\phibf^{(n+1)}-\phibf^{(n)}] 
+\frac12\, {\rm D}^2\calE(\xibf)[\phibf^{(n+1)}-\phibf^{(n)}, \phibf^{(n+1)}-\phibf^{(n)}] 
\end{align*}
for some $\xibf$ in the neighborhood of the ground state. Estimating the derivative  
\begin{align*}
{\rm D}\calE(\phibf^{(n)})&[\phibf^{(n+1)}-\phibf^{(n)}]
= a_{\phibf^{(n)}}\bigl(\phibf^{(n)},\phibf^{(n+1)}-\phibf^{(n)}\bigr)   \\
&\le \tau\,a_{\phibf^{(n)}} \bigl(\phibf^{(n)},\etabf^{(n)}\bigr) 
+ \|\phibf^{(n)}\|_{a_{\phibf^{(n)}}} 
\| \calR\bigl(\phibf^{(n)},\tau\,\etabf^{(n)}\bigr) - \big(\phibf^{(n)} + \tau\,\etabf^{(n)} \big) \|_{a_{\phibf^{(n)}}} \\
&\le -\tau\, \|\etabf^{(n)}\|^2_{a_{\phibf^{(n)}}} + C_H^2\,\tau^2\,  \|\phibf^{(n)}\|_{a_{\phibf^{(n)}}} 
\|\phibf^{(n)} + \tau\, \etabf^{(n)}\|_{a_{\phibf^{(n)}}} \|\etabf^{(n)}\|_{a_{\phibf^{(n)}}}^2 
\end{align*}
and using \eqref{eq:SecDerBound} together with \eqref{eq:UniformBound}, 
%\begin{align*}
%{\rm D}^2\calE(\xibf)[\phibf^{(n+1)}-\phibf^{(n)}, \phibf^{(n+1)}-\phibf^{(n)}]
%& \le  C_0\, \|\phibf^{(n+1)}-\phibf^{(n)}\|_{a_{\phibf^{(n)}}}^2 \\
%& \le  2C_0\big(1+\tau_{\max}^2C_1^2\big)\, \tau^2 \|\etabf^{(n)}\|^2_{a_{\phibf^{(n)}}},
%\end{align*}
%
%\begin{align*}
%{\rm D}\calE(\phibf^{(n)})[\phibf^{(n+1)}-\phibf^{(n)}]
%= a_{\phibf^{(n)}}\bigl(\phibf^{(n)},\phibf^{(n+1)}-\phibf^{(n)}\bigr)  
%&= \tau_n\,a_{\phibf^{(n)}} \bigl(\phibf^{(n)},\etabf^{(n)}\bigr) + O(\tau^2)\\
%&= -\tau_n\,a_{\phibf^{(n)}}\bigl(\etabf^{(n)},\etabf^{(n)}\bigr) + O(\tau^2)
%\end{align*}
%
%which follows again from Proposition~\ref{prop:retr_est}, 
we conclude that
\begin{align*}
\calE(\phibf^{(n)}) - \calE(\phibf^{(n+1)})
&= - {\rm D}\calE(\phibf^{(n)})[\phibf^{(n+1)}-\phibf^{(n)}] 
- \frac12\, {\rm D}^2\calE(\xibf)[\phibf^{(n+1)}\!-\!\phibf^{(n)}, \phibf^{(n+1)}\!-\!\phibf^{(n)}] \\
%
%&\ge \tau_n\, \|\etabf^{(n)}\|^2_{a_{\phibf^{(n)}}}
%- \tau_n^2\ \|\phibf^{(n)}\|_{a_{\phibf^{(n)}}} 
%\|\phibf^{(n)} + \tau_n\, \etabf^{(n)}\|_{a_\phibf^{(n)}} \|\etabf^{(n)}\|^2_{a_{\phibf^{(n)}}} \\
%&\qquad- \tfrac 12\, C\, \|\phibf^{(n+1)}-\phibf^{(n)} \|^2_{a_{\phibf^{(n)}}} \\
%
&\ge \tau\, \|\etabf^{(n)}\|^2_{a_{\phibf^{(n)}}}
- C_H^2\,\tau^2\, \|\phibf^{(n)}\|_{a_{\phibf^{(n)}}} 
\|\phibf^{(n)} + \tau\, \etabf^{(n)}\|_{a_{\phibf^{(n)}}} \|\etabf^{(n)}\|^2_{a_{\phibf^{(n)}}} \\
&\qquad -2\,C_0\, \tau^2\, \|\etabf^{(n)}\|^2_{a_{\phibf^{(n)}}}  
\!- 2\,C_0\,C_H^2\, \tau^4\,\|\phibf^{(n)}+ \tau \etabf^{(n)}\|^2_{a_{\phibf^{(n)}}} \|\etabf^{(n)}\|^4_{a_{\phibf^{(n)}}} \\
%
%&\ge \tau\, \|\etabf^{(n)}\|^2_{a_{\phibf^{(n)}}}
%- \tau^2\ C_2\, \|\etabf^{(n)}\|^2_{a_{\phibf^{(n)}}} 
%- \tau^4\, C_3\, \|\etabf^{(n)}\|^2_{a_{\phibf^{(n)}}} \\
&\ge \tau\, \|\etabf^{(n)}\|^2_{a_{\phibf^{(n)}}}\ 
\big( 1-\tau_{\max}\, C_3 -\tau_{\max}^3 C_4 \big)
\end{align*}
with $C_3=C_H^2C_2^2(1+\tau_{\max}C_1)+2\,C_0$ and $C_4=2\,C_0 C_H^2 C_1^2C_2^4(1+\tau_{\max}C_1)^2$.
Finally, it follows from \eqref{eq:est1} and \eqref{eq:UniformBound} that 
$$
\tau\,\|\etabf^{(n)}\|_{a_{\phibf^{(n)}}} \ge \frac{\|\phibf^{(n+1)}-\phibf^{(n)}\|_{a_0}}{1+\tau_{\max}C_5},
$$
with $C_5=C_H^2C_1C_2^2(1+\tau_{\max}C_1)$. Thus, we obtain the estimate~\eqref{eq:energy_decr} for the sufficiently small step size 
$0<\tau \le \tau_{\max}$ and a~constant $C_\text{D}>0$ depending on $\tau_{\max}$ and the other constants only. 
\end{proof}	
\begin{lemma}[Sufficient condition for~\ref{A3}]\label{lem:condA3}
Consider a~sufficiently small step size\linebreak $0<\tau_{\min}\le\tau\le\tau_{\max}$. If the iterates~$\phibf^{(n)}$ given by~\eqref{eq:RiemannianGradMethod} with the polar decomposition based retraction~\eqref{eq:retr_pol} are in the neighborhood of the ground state, then there exists a~constant $C_\text{S}>0$ such that %for large enough~$n$, we have
the estimate \eqref{eq:step_size} is satisfied.
%\begin{equation}\label{eq:step_size}
%	\|\phibf^{(n+1)}-\phibf^{(n)}\|_{a_{0}} \geq C_\text{S}\, \|\grad\,\calE(\phibf^{(n)})\|_{a_{\phibf^{(n)}}}.
%\end{equation}
\end{lemma}
\begin{proof}	
Using \eqref{eq:RelDifGrad} and \eqref{eq:UniformBound}, we estimate
\begin{align*}
\tau\,\|\etabf^{(n)}\|_{a_{\phibf^{(n)}}} 
&\le \|\phibf^{(n+1)}-\phibf^{(n)}\|_{a_{\phibf^{(n)}}} +\tau^2\,\|\phibf^{(n)} + \tau\, \etabf^{(n)}\|_{a_{\phibf^{(n)}}} \|\etabf^{(n)}\|^2_H \\
&\le \|\phibf^{(n+1)}-\phibf^{(n)}\|_{a_{\phibf^{(n)}}} +C_5\tau^2\,\|\etabf^{(n)}\|_{a_{\phibf^{(n)}}} .
\end{align*}
Therefore, a step size restriction $0<\tau_{\min}\le \tau\le\tau_{\max}$ with sufficiently small $\tau_{\max}$ yields 
\begin{align*}
\|\phibf^{(n+1)}-\phibf^{(n)}\|_{a_0}
& \ge c_E\, \|\phibf^{(n+1)}-\phibf^{(n)}\|_{a_{\phibf^{(n)}}} 
  \ge c_E\,(1-\tau C_5)\, \tau\, \|\etabf^{(n)}\|_{a_{\phibf^{(n)}}} \\
& \ge c_E\,(1-\tau_{\max}C_5)\,\tau_{\min} \|\etabf^{(n)}\|_{a_{\phibf^{(n)}}}
=C_\text{S}\,\|\grad\,\calE(\phibf^{(n)})\|_{a_{\phibf^{(n)}}}
\end{align*}
with $C_\text{S}=\tau_{\min}c_E(1-\tau_{\max}C_5)>0$.
\end{proof}
\begin{remark}
Note that in Lemma~\ref{lem:condA2} and Lemma~\ref{lem:condA3}, the polar decomposition based retraction can be replaced by the $qR$-based retraction defined in \eqref{eq:retr_qR} or any other second-order bounded retraction. 
%This immediately follows from Proposition~\ref{prop:retr_est_qR}.
\end{remark}
%
%=============================================================================
%
\subsection{Step size control with a non-monotone line search}\label{sect:Riemannian:lineSearch}
In order to accelerate the convergence of the Riemannian gradient descent method~\eqref{eq:RiemannianGradMethod}, we determine the step size by employing the non-monotone line search algorithm~\cite{ZhaW04} combined with the alternating Barzilai-Borwein step size strategy as proposed in~\cite{WenY13}. The resulting Riemannian gradient descent method is presented in Algorithm~\ref{alg:nonmonGradient}.   
\begin{algorithm}[h!] 
	\setstretch{1.15}
	\caption{Riemannian gradient descent method with non-monotone line search}
	\label{alg:nonmonGradient}
	\begin{algorithmic}[1]
		\State{\textbf{Input:} energy~$\calE$, retraction~$\calR$, initial guess~$\phibf^{(0)}\in\Stiefel$, 
			$c_0 = \calE(\phibf^{(0)})$, $q_0=1$,\\ \hspace{3.2em} parameters~$\alpha \in [0,1]$,
			$\beta, \delta\in(0,1)$, $0<\gamma_{\min}<\gamma_{\max}$, $\gamma_0>0$}
		\vspace{0.5em}	
		\For{$n=0,1,2,\dots$}
		\State{Compute a~search direction~$\etabf^{(n)}$ as an~approximation of $-\grad\calE(\phibf^{(n)})$.}
		\If{$n>0$} \State{Compute a trial step size 
			$$
			\gamma_n = \left\{
			\begin{array}{cl}
			\frac{(\sbf^{(n)},\sbf^{(n)})_H}{|(\sbf^{(n)},\ybf^{(n)})_H|}& \mbox{for odd}~~n,\\[2mm]
			\frac{|(\sbf^{(n)},\ybf^{(n)})_H|}{(\ybf^{(n)},\ybf^{(n)})_H}& \mbox{for even}~~n,
			\end{array}
			\right.
			$$
			\hspace*{11mm} where $\sbf^{(n)} = \phibf^{(n)} - \phibf^{(n-1)}$ and 
			%			$\ybf^{(n)} =\grad\,\calE(\phibf^{(n)}) -  \grad\,\calE(\phibf^{(n-1)})$
			$\ybf^{(n)} = \etabf^{(n-1)} - \etabf^{(n)}$.}
		\EndIf
		\State{Set $\gamma_n=\max(\gamma_\mathrm{min},\min(\gamma_n,\gamma_\mathrm{max}))$.}
		\State{Find the smallest $k\in\N$ such that~$\tau_n=\gamma_n\delta^k$ satisfies the non-monotone condition $$\calE\bigl(\calR(\phibf^{(n)},\tau_n\etabf^{(n)})\bigr)\leq c_n
			-\beta\, \tau_n\, a_{\phibf^{(n)}}( \etabf^{(n)},\etabf^{(n)}).$$ 	
			%+\beta\, \tau_n\, a_{\phibf^{(n)}}(\grad\,\calE(\phibf^{(n)}),\etabf^{(n)}).$$
		}
		%\State{Set $\tau_n=\gamma_n \delta^k$.}  \label{alg:nmg:gamma}
		\State{Set $\phibf^{(n+1)} = \calR(\phibf^{(n)},\tau_n \etabf^{(n)})$.}
		\State{Compute $q_{n+1}=\alpha q_{n} +1$ and  $c_{n+1}=\bigl(1-\tfrac{1}{q_{n+1}}\bigr)c_n+\tfrac{1}{q_{n+1}}\calE(\phibf^{(n+1)})$.
		}
		\EndFor
		\vspace{0.5em}	
		\State \textbf{Output:} sequence of iterates $\{\phibf^{(n)}\}$
	\end{algorithmic}
\end{algorithm}

The following theorem establishes that a convergent sequence generated by this algorithm yields a~stationary point.
\begin{theorem}
	Let $\{\phibf^{(n)}\}$ be a~sequence generated by Algorithm~\textup{\ref{alg:nonmonGradient}}. Then every accumulation point $\phibf^*$ of this sequence is a~critical point of $\calE$, i.e., we have $\grad\calE(\phibf^*) = 0$.
\end{theorem}
\begin{proof}
	Since the retractions considered in Section~\ref{sec:retraction} are globally defined, the result can be proved analogously to \cite[Th.~3.3]{HuLWY20}.
\end{proof}
%
%
%=============================================================================
%
\subsection{Inexact gradient descent schemes}
In this subsection, we propose an inexact gradient descent method which significantly reduces the computational complexity of the iteration~\eqref{eq:RiemannianGradMethod}.

First, we establish a connection of our minimization method to the DCM method considered in~\cite{SchRNB09}. Let $\phibf^*\in\Stiefel$ be a~critical point of $\calE$, 
i.e., \mbox{$\grad\calE(\phibf^*)=0$}. Then \eqref{eq:gradE} yields%
\begin{equation}\label{eq:stationary}
  \calA_{\phibf^*}\, \phibf^* 
  = \phibf^*\, \out{\phibf^*}{\calA_{\phibf^*}^{-1}\phibf^*}^{-1}.
\end{equation}
This equation further implies
\begin{equation}\label{e:matequiv}
  \out{\phibf^*}{\calA_{\phibf^*}\phibf^*} 
  = \out{\phibf^*}{\calA_{\phibf^*}^{-1}\phibf^*}^{-1}
\end{equation}
and, hence, \eqref{eq:stationary} can be rewritten as $\calA_{\phibf^*} \phibf^*=\phibf^*\,\out{\phibf^*}{\calA_{\phibf^*}\phibf^*}$. In the DCM method considered in~\cite{SchRNB09}, the search direction is taken as 
\begin{align}
\label{etan:precond:precondDCM}
  -\calB_\phibf^{-1} \big(\calA_\phibf \phibf - \phibf\, \out{\phibf}{\calA_\phibf \phibf} \big), 
\end{align}
where $\calB_\phibf$ is a~given preconditioner. Without $\calB_\phibf$, this leads to the Riemannian gradient descent method in the Hilbert metric~$g_H$, which usually shows slow convergence. Considering the preconditioner~$\calB_\phibf = \calA_\phibf$ yields the search direction $-\phibf + \calA_\phibf^{-1}\phibf\, \out{\phibf}{\calA_\phibf \phibf}$. Due to \eqref{e:matequiv}, this search direction is asymptotically equivalent to 
$$
\etabf=-\grad\calE(\phibf) = -\phibf + \calA_\phibf^{-1}\phibf\, \out{\phibf}{\calA_\phibf^{-1}\phibf}^{-1}.
$$ 
This observation shows that a suitably preconditioned DCM admits a near gradient descent structure in the novel metric $g_a$. 

The computation of both search directions requires the solution of a system involving the operator~$\calA_\phibf$ in each step but different linear combinations of the outcome are used. For the DCM, it is known that an approximation of ~$\calA_\phibf$ is sufficient for convergence in practice. In this spirit, we may also use the inexact gradient. This consideration motivates to use 
\begin{align}
\label{etan:precond:RGD}
	-\grad\calE(\phibf)
	\approx -\phibf + \calB^{-1}_\phibf\phibf\, \out{\phibf}{\calB^{-1}_\phibf\phibf}^{-1}
\end{align}
as a search direction. Here, $\calB_\phibf\approx \calA_\phibf$ is a suitable preconditioner that realizes, e.g., a few iterations of a preconditioned iterative solver for $\calA_\phibf^{-1}\phibf$ with starting value 
$$
\phibf\, \out{\phibf}{\calA_{\phibf} \phibf}^{-1} \approx \calA_{\phibf}^{-1}\phibf.
$$
The error of the proposed starting value is roughly as accurate as the current approximation of the wavefunction in the iteration. Hence, after only a few steps of the preconditioned iterative solver the residual of the linear system is substantially smaller than the current error. In a convergent iteration, sufficiently many (inner) iterations will guarantee that the resulting direction is a descent direction, cf.~the numerical experiments of Section~\ref{sect:examples}. The control of the number of iterations required to ensure a descent could be integrated into the method. 
%
%=============================================================================
%

%=============================================================================
%=========  Examples
%=============================================================================
\section{Examples}\label{sect:examples}
In this final section, we present two examples which fit in the framework of Section~\ref{sect:generalModel}. Moreover, the efficiency of the proposed algorithm (and its preconditioned variants) are illustrated in a number of numerical experiments. 
%
%=============================================================================
%
\subsection{Gross-Pitaevskii eigenvalue problem}
In the special case $N=1$, we seek an~eigenfunction $u\in V := H^1_0(\Omega)$ satisfying the normalization constraint $\|u\|_{L^2(\Omega)} = 1$. 
%the constraint of being on the Stiefel manifold reduces to a simple . 
Hence, the minimization takes place on the unit sphere $\mathbb{S}=\{v\in V\ :\ \|u\|_{L^2(\Omega)} = 1\}$. 
A well-known example, which fits in this framework, is the {\em Gross-Pitaevskii eigenvalue problem}. In the classical form, this reads
\[
	-\Delta u + V_\text{ext}\, u + \kappa\, |u|^2 u = \lambda\, u 
\] 
for some non-negative and space-dependent external potential~$V_\text{ext}\ge0$ and a constant $\kappa\ge 0$ regulating the strength of the nonlinearity. Here, the bilinear form $a_u\colon V\times V\to \R$ is given by  
\[
	a_u(v,w)
	:= \int_{\Omega} \nabla v\cdot \nabla w + V_\text{ext}\, v\, w + \kappa\, |u|^2 v\, w \dx. 
\]
The linear part~$a_0$, which contains the weak Laplacian and the potential, defines an inner product on $V$. For the nonlinear part, we set~$\gamma(\rho(u)) = \gamma(|u|^2) := \kappa\, |u|^2$, i.e., a constant times the density of~$u$. Hence, for any $u\in V$, the bilinear form~$a_u$ defines an inner product on $V$ and Assumption~\ref{ass:aPhi} is satisfied. Due to $\Gamma(\rho) = \kappa \int_0^\rho t \dt = \frac 12 \kappa \rho^2$, the corresponding energy has the form 
\[	
	\calE(u)
	= \frac12\, a_0(u,u) + \frac12\, \int_{\Omega} \Gamma(\rho(u)) \dx
%	= \frac12\, a_0(u,u) + \frac\kappa4\, \int_{\Omega} |u|^4 \dx
	= \frac12\, \int_{\Omega} \|\nabla u\|^2 + V_\text{ext}\, |u|^2 + \frac\kappa2\, |u|^4 \dx.
\]
The assumed property that $\calE$ does not change if the argument is multiplied by an~orthogonal matrix translates in the case $N=1$ to $\calE(\pm u)=\calE(u)$, which is clearly satisfied. As before, we are interested in the ground state, i.e., the state of minimal energy. For the Gross-Pitaevskii eigenvalue problem, the ground state coincides with the eigenfunction that corresponds to the smallest eigenvalue.   

% A-method
Following the procedure presented in Section~\ref{sect:generalModel}, we have $\out{u}{v}=(u,v)_H=(u,v)_{L^2(\Omega)}$ and $T_u\,\mathbb{S} =\{v\in V\ :\ (u,v)_H = 0\}$. Hence, the normal space is one-dimensional, and~\eqref{property:Psi} reduces to find $\psi \in V$ such that 
\begin{align*}
	a_u(\psi, v) = 0 \quad\text{for all } v\in T_u\,\mathbb{S}, \qquad
	(\psi, u)_H = 1.
\end{align*}
Written as a saddle point problem, we seek $(\psi,\mu) \in V\times \R$ such that
\begin{subequations}
\label{eq:psiGP}
\begin{align}
	a_u(\psi, v) &= \mu\, (u, v)_H \hspace{1.8em}\text{ for all } v\in V, \label{eq:psiGP1}\\
	(\psi, u)_H &= 1.\label{eq:psiGP1}
\end{align}
\end{subequations}
The resulting projection applied to $u$ reads~$P_u(u) = u - \psi$. For $N=1$, the polar decomposition based retraction from Section~\ref{sec:retraction:polar} as well as the $qR$-based retraction from Section~\ref{sec:retraction:qR} simply equal a  $L^2$-normalization. This then leads to the following iteration scheme: Given $u^{(n)}\in\mathbb{S}$, compute $\psi^{(n)} = \psi(u^{(n)})$ by solving \eqref{eq:psiGP} with $u=u^{(n)}$ and set 
\[
	\tilde u^{(n+1)} := (1-\tau_n) u^{(n)} + \tau_n \psi^{(n)},\qquad
	u^{(n+1)} := \frac{\tilde u^{(n+1)}}{\| \tilde u^{(n+1)}\|_{L^2(\Omega)}}.
\]
Note that this is exactly the damped GF$a_z$ method introduced in~\cite{HenP20}, which is labeled $A$-method in \cite{AltHP21}. Moreover, in the special case $\tau_n\equiv 1$, this iteration is the straight-forward generalization of the {\em inverse power method} to the nonlinear setting. We refer to the aforementioned original papers as well as to \cite{AltP19,altmann2020localization} for numerical experiments that demonstrate the competitiveness of the method with established schemes and its ability to capture relevant physical phenomena such as the exponential localization of eigenstates. The guaranteed energy decay of the method has also been exploited explicitly in \cite{doi:10.1137/20M1332864}.
%
%
%=============================================================================
%
\subsection{Kohn-Sham model}
A second example, which is covered by this paper, is the Kohn-Sham model~\cite{KohS65} and, in particular, the model based on the {\em density functional theory}~\cite{HohK64}. This theory allows a reduction of the degrees of freedom, leading to a model which balances accuracy and computational cost, see also~\cite{YanMLW09,CanCM12,CanDMSV16} for a more detailed introduction.
%
%
%=======================================================================
\subsubsection{Validation of the model}
% KS model
As an~energy functional, we consider (with $\Omega=\R^3$) 
\begin{align}
	\calE(\phibf)
	&= \frac 12\, \sum_{j=1}^N \int_{\Omega} \|\nabla \phi_j(r)\|^2 \dr
	+ \int_{\Omega} V_\text{ion}(r)\, \rho(\phibf(r)) \dr \notag \\
	&\hspace{2.2cm}+ \frac 12 \int_{\Omega}\int_{\Omega} \frac{\rho(\phibf(r))\, \rho(\phibf(r'))}{\|r-r'\|} \dr\drp 
	+ \int_{\Omega} \epsilon_\text{xc}(\rho(\phibf(r)))\, \rho(\phibf(r)) \dr \label{eq_energyKS}
\end{align}
with the ionic potential~$V_\text{ion}$, the exchange-correlation~$\epsilon_\text{xc}$, and the associated electronic charge density~$\rho(\phibf(r)) = \phibf(r) \cdot \phibf(r) = \sum_{j=1}^N |\phi_j(r)|^2$. Based on semi-empirically knowledge of the model, the particular exchange-correlation is described in~\cite{YanMLW09}. For more details, on this and the corresponding~{\em local density approximation}, we refer to~\cite{PerZ81,MacCMD91}. 
Following the physical setup, the ionic potential typically reads 
\[
  V_\text{ion}(r)
  = \sum_{j=1}^{N_\text{nuc}} \frac{z_j}{\| r-r_j\|}
\]
with the number of nuclei~$N_\text{nuc}$, the charge of the $j$th nuclei~$z_j$, and its position~$r_j$, which are assumed to be fixed. The obvious problem of the included singularities can be circumvented by considering core electrons (which are very close to a nucleus) as part of the corresponding core. For more details on this so-called {\em pseudopotential approximation}, we refer once more to~\cite{YanMLW09} and the references therein. As a consequence, we may assume in the following that~$V_\text{ion}$ in~\eqref{eq_energyKS} is a bounded potential. 

% eigenvalue problem
We are interested in the Kohn-Sham ground state, which means that we aim to minimize the energy~$\calE$ over~$V=\tilde{V}^N$ with $\tilde{V}=H^1_\text{per}(\Omega)$, i.e., the Sobolev space $H^1(\Omega)$ with periodic boundary conditions, subject to the constraint~$\out{\phibf}{\phibf}=\Onebf$. Hence, the minimization takes place on the Stiefel manifold $\Stiefel$. Following~\eqref{eq:aphi}, the corresponding bilinear form reads 	
\[
  a_\phibf(\vbf,\wbf)
%  = \int_\Omega \nabla\vbf\cdot\nabla\wbf \dr + 2\, \int_\Omega V_\text{ion}\, \vbf\cdot\wbf \dr
%  + \int_{\Omega}\int_\Omega \frac{\rho(\phibf(r')) \, \vbf\cdot\wbf}{|r-r'|} \drp\dr + 2\, \int_\Omega  \mu_{\rm xc}(\rho(\phibf)) \, \vbf\cdot\wbf \dr
  = \int_\Omega \tr\bigl((\nabla \vbf)^T\nabla\wbf\bigr) \dr 
  + 2\int_\Omega  V_\text{ion}\, \vbf\cdot\wbf \dr
  + \int_{\Omega} \gamma(\rho(\phibf))\, \vbf\cdot \wbf \dr
\]
with the (non-local) nonlinearity 
\[
%  \Gamma(\rho) 
%  := \int_{\Omega} \frac{\rho\cdot\rho(\phibf(r'))}{|r-r'|} \drp 
%  + 2\,\epsilon_\text{xc}(\rho)\, \rho,
%
  \gamma(\rho) 
  = 2\int_{\Omega} \frac{\rho(\phibf(r'))}{\|r-r'\|} \drp 
  + 2\, \frac{\rm d}{{\rm d}\rho} \big(\rho\,\epsilon_{\rm xc}(\rho)\big).
%
%  \mu_{\rm xc}(\rho) = \frac{\rm d}{{\rm d}\rho} \big(\rho\,\epsilon_{\rm xc}(\rho)\big).
\]
\begin{lemma}
\label{lem:Garding}
Consider a fixed~$\phibf\in V$ and assume that $V_\text{ion}$ and $\gamma(\rho(\phibf))$ are bounded. Then the corresponding bilinear form
\[
  \tilde a_\phibf(v_j, w_j)
  = \int_\Omega (\nabla v_j)^T\nabla w_j \dr 
  + 2 \int_\Omega V_\text{ion}\, v_j\, w_j \dr
  + \int_{\Omega} \gamma(\rho(\phibf))\, v_j\, w_j \dr
\]
satisfies a G\aa rding inequality. Hence, there exists~$\sigma\in\R$ such that~$\tilde a_\phibf + \sigma\,(\,\cdot\,, \cdot\,)_{L^2(\Omega)}$ is a~symmetric, bounded, and coercive bilinear form. 
\end{lemma}
\begin{proof}
Let $c_V$ and $c_\gamma$ denote the (possibly negative) lower bounds of~$2V_\text{ion}$ and $\gamma(\rho(\phibf))$, respectively. Then, the definition of~$\tilde a_\phibf$ gives
\[
  \tilde a_\phibf(v, v)
  \ge \int_\Omega \|\nabla v\|^2 \dr + (c_V + c_\gamma)\, (v, v)_{L^2(\Omega)}  
  = \| v\|_{\tilde V}^2 - (1 - c_V - c_\gamma)\, \|v\|_{L^2(\Omega)}^2 
\]
for all $v\in \tilde V$. The coercivity of $\tilde a_\phibf + \sigma\,(\,\cdot\,, \cdot\,)_{L^2(\Omega)}$ then follows for any~$\sigma \ge 1 - c_V - c_\gamma$. Symmetry and boundedness are directly given. 
\end{proof}
Since we cannot ensure that $\tilde a_\phibf$ is coercive, we need to adapt the original nonlinear eigenvector problem~\eqref{eq:EVPweakComponents} by a {\em shift}: seek~$\phibf\in\Stiefel$ and $\lambda_1,\dots,\lambda_N\in\R$ such that
\begin{align*}
  \tilde a_\phibf(\phi_j, v_j) + \sigma\, (\phi_j, v_j)_{L^2(\Omega)}
  = (\lambda_j+\sigma)\, (\phi_j, v_j)_{L^2(\Omega)} \qquad\text{ for all }(v_1, \dots, v_N)\in V
\end{align*}
with the shift~$\sigma$ from Lemma~\ref{lem:Garding}. This gives a coupled system of nonlinear eigenvector problems, which satisfies Assumption~\ref{ass:aPhi} and, therefore, the theory of this paper is applicable.
%
%
%=======================================================================
\subsubsection{Numerical experiments}
We now illustrate the convergence behaviour of the new energy-adaptive Riemannian gradient descent scheme (RGD) and its variants and show that they are competitive with the established SCF iteration and the preconditioned DCM method. 
The numerical experiments are performed on an Intel(R) Core(TM) i7-8565U CPU@1.80GHz using MATLAB (version R2021b). The implementation is based on the MATLAB toolbox KSSOLV, cf.~\cite{YanMLW09}. The usage of this toolbox allows us to focus on the new eigenvalue iterations and their comparison to already existing methods. Note that the toolbox works with an additional factor of two in the electronic charge density. This, however, does not affect the convergence behaviour. We initially select an exemplary molecule system implemented in KSSOLV, namely $CO_2$ \mbox{($N=8$)}. 
In KSSOLV, a spatial discretization using a planewave discretization of functions in~$V := [H^1(\R^3)]^N$ is considered. As in \cite{YanMLW09},  
% kinetic energy cutoff of 25 
% 10×10×10 (atomic units) cubic supercell, 
we use a~$32\times 32\times 32$ sampling grid for the wavefunctions in the $CO_2$ model. 
% CO2
% navigate to codes/kssolv-toolbox/kssolv
% start matlab
% >> co2_setup % mol contains information about molecule and spatial discretization
% >> options = setksopt; % options contains information about options, in particular, tolerances for iterations 
% >> options.maxearinneriter = 3; % manually set number of pcg steps for approximate gradient computation, large number corresponds to "exact solve" up to tolerance options.cgtol
% >> [Etotvec, X, vtot, rho] = ear(mol,options); % run Energy-Adaptive Riemannian gradient descent, number of iterations and total time will be dispayed, starting value according to 1 iteration of scf, linear solver to be determined around line 178, pcg does not always work, minres appears to be robust but slightly slower
% >> [Etotvec, X, vtot, rho] = scf(mol,options); % run Self-Consisten-Field iteration for comparison, number of iterations and total time will be dispayed

We shall first illustrate the convergence behaviour of the RGD. For the $CO_2$ molecule, we compare the following variants:
\begin{itemize}
	\item RGD from~\eqref{eq:RiemannianGradMethod} for several choices of a constant time step size~$\tau_n=\tau$ with $\tau\in\{0.05,0.1,0.15,0.2\}$,  
	\item RGD with the non-monotone line search as presented in Algorithm~\ref{alg:nonmonGradient} with the descent direction $\etabf^{(n)} = -\grad\calE(\phibf^{(n)})$ and parameters~$\alpha = 0.95$, $\beta = 10^{-4}$,  $\gamma_\text{min} = 10^{-4}$, $\gamma_\text{max} = 1.0$, $\gamma_0 = 10^{-2}$, and $\delta = 0.5$.
\end{itemize}
For both variants, the polar decomposition based retraction~\eqref{eq:retr_pol} is used. As a~stopping criterion, we consider the $H$-norm of the residual to fall below the tolerance $\tol=10^{-6}$. The linear systems are solved up to the higher accuracy of $10^{-8}$.

Figure~\ref{fig:convergence}~(left) shows the evolution of the residuals in the iteration. In accordance with the theoretical predictions, we observe convergence for sufficiently small constant step sizes. A look into the corresponding errors in the energy (with respect to a reference minimal energy computed to higher accuracy) depicted in Figure~\ref{fig:convergence}~(right) shows that for $\tau=0.2$, after an initial decay, the method approaches some other critical point on a~higher energy level. Furthermore, for smaller choices of $\tau$, the linear convergence to the ground state is observed. There is probably an optimal choice of the step size around~$\tau=0.15$ that minimizes the linear rate of convergence. However, the non-monotone line search converges much faster and appears to be much more efficient for this example and many others that we have tried. 
\begin{figure}
\input{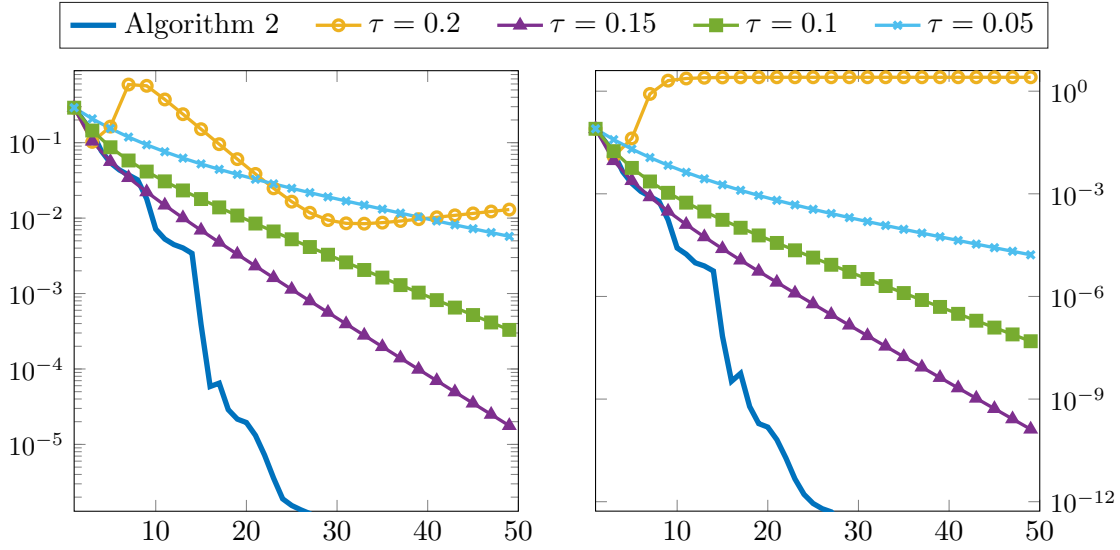}
\caption{Convergence history of the residual (left) and energy (right) for the $CO_2$ model for different (fixed) step sizes and the non-monotone line search from Algorithm~\ref{alg:nonmonGradient}.}
\label{fig:convergence}
\end{figure}

While the line search optimizes the iteration count, the cost per iteration step is largely reduced by the inexact solution of linear system for the gradient computation in each step. We will refer to the corresponding scheme as:
\begin{itemize}
	\item inexact RGD with the non-monotone line search as presented in Algorithm~\ref{alg:nonmonGradient} with~$\etabf^{(n)}$ being the preconditioned MINRES approximation given in~\eqref{etan:precond:RGD}. %of $-\grad\calE(\phibf^{(n)})$ 
	We use the MINRES implementation of MATLAB using the KSSOLV built-in Teter preconditioner \cite{Teter89,YanMLW09}.
\end{itemize}
We also compare the performance of the exact and inexact RGD with the established schemes
\begin{itemize}
	\item SCF: self-consistent field iteration as readily available in KSSOLV using LOPCG to solve the linear eigenvalue problem in each step up to tolerance $10^{-8}$, 
	\item DCM: direct constrained minimization as defined in \eqref{etan:precond:precondDCM} with non-monotone line search. 
	%(using the same parameters as RGD above). 
	The preconditioner is given by $3$ steps of the preconditioned MINRES iteration as in the inexact RGD. 
\end{itemize}
All schemes use the same initial guess to the wavefunction and the $qR$-based retraction defined in \eqref{eq:retr_qR}. 
For the RDG variants and DCM, we use the non-monotone line search with the prescribed parameters given above.

Table~\ref{tab:runtimes} shows the CPU times and (outer) iteration counts of the four methods. While solving the linear systems too accurately seems to be suboptimal in terms of computational complexity, the numbers clearly indicate that the inexact RGD substantially accelerates the simulation and is very competitive with SCF. The closely related preconditioned DCM variant performs equally well asymptotically but, according to our experience, is a bit slower in the initial phase when the residuals are still large. 
\newcolumntype{P}{C{2.2cm}}
\begin{table}[h]
	\caption{CPU time (in seconds) and number of needed iteration steps to achieve an approximation of the ground state of the $CO_2$ molecule with tolerance $10^{-6}$ in the residual. }\label{tab:runtimes}
	\begin{tabular}{r||P|P|P|P|}
		& SCF & RGD & inexact RGD & prec.~DCM 
		\\[0.2em] \hline 
		CPU time & $12.3$  & $36.7$  & $11.6$  & $16.6$  \\[0.2em]
		\# iterations & $8$ & $28$ & $37$ & $45$
	\end{tabular}
\end{table}

According to our experience, the competitiveness of inexact RGD is representative. An experiment for the more challenging molecule petacene, also implemented in KSSOLV, supports this assessment; see Table~\ref{tab:runtimespenta}. Due to the large number of electrons in pentacene ($N=102$), a~sampling grid of size~$64\times 32\times 48$ is used for the  spatial planewave discretization. 

\newcolumntype{P}{C{2.2cm}}
\begin{table}[t]
	\caption{CPU time (in seconds) and number of needed iteration steps to achieve an approximation of the ground state of the {\em pentacene} molecule with tolerance $10^{-6}$ in the residual. }\label{tab:runtimespenta}
	\begin{tabular}{r||P|P|P|}
		& SCF  & inexact RGD & prec.~DCM 
		\\[0.2em] \hline 
		CPU time & $3211$  & $2204$  & $2655$  \\[0.2em]
		\# iterations & $14$ & $43$ & $51$
	\end{tabular}
\end{table} 

%
%
%=============================================================================
%=========  Conclusion
%=============================================================================
\section{Conclusion}\label{sect:conclusion}
In this paper, we have generalized the energy-adaptive gradient descent scheme from \cite{HenP20} to nonlinear eigenvector problems formulated on the infinite-dimensional Stiefel manifold. We have shown convergence of the method and a guaranteed energy decay of the iterates if the step size is sufficiently small. 
Moreover, we have introduced a non-monotone step size control and discussed the inexact variants, which accelerate the proposed method significantly. In total, this gives a novel energy-adaptive descent scheme, which is competitive with existing schemes such as SCF and DCM. 

\section*{Acknowledgement} The authors would like to thank Patrick Henning for his inspiring work on the energy-adaptive Riemannian gradient descent method for the Gross-Pitaevskii problem and Benjamin Stamm for raising the question of its applicability to the Kohn-Sham model.  	
%
%
%=============================================================================
%=========  Bibliogrphy
%=============================================================================
\bibliographystyle{alpha}
\bibliography{refKS}
%
%
%=============================================================================
%=========  Appendix 
%=============================================================================
\appendix
\section{Proofs of Second-order Bounds for the Retractions}
%
%
%=======================================================================

\subsection{Proof of Proposition~\ref{prop:retr_est}}\label{app:proof:projective}
%\begin{proposition}\label{prop:retr_est}
%	The retraction $\calR$ in \eqref{eq:retr_pol} satisfies
%	\begin{equation}
%	\|\calR(\phibf, t\etabf)-(\phibf+t\etabf)\|_{a_\phibf}\leq 
%	t^2\,\|\phibf+ t\etabf\|_{a_\phibf}\|\etabf\|_H^2.
%	\end{equation}
%\end{proposition}
%
%\begin{proof}[Proof of Proposition~\ref{prop:retr_est}]
	For
	$$
	\wbf(t):=\calR(\phibf, t\etabf)-(\phibf+t\etabf)=(\phibf+t\etabf)\bigl((\Onebf+t^2\out{\etabf}{\etabf})^{-1/2}-\Onebf\bigr)
	$$
	we have
	$$
	\|\wbf(t)\|_{a_{\phibf}}
	\leq \|\phibf+t\etabf\|_{a_{\phibf}}\,  \|(\Onebf+t^2\out{\etabf}{\etabf})^{-1/2}-\Onebf\|_2,
	$$
	where $\|\cdot\|_2$ denotes the spectral matrix norm. Let $\mu_1\geq\ldots\geq\mu_N>0$ be the eigenvalues of the symmetric, positive semidefinite matrix $\out{\etabf}{\etabf}$. Then, 
	$$
	\|(\Onebf+t^2\out{\etabf}{\etabf})^{-1/2}-\Onebf\|_2
	=\max_{1\leq j\leq N} \Bigl(1-\frac{1}{\sqrt{1+t^2\mu_j }}\Bigr)
	=1-\frac{1}{\sqrt{1+t^2\mu_1}}.
	$$
	By the mean value theorem, there exists $\theta\in (0, t)$ such that 
	$$
	1-\frac{1}{\sqrt{1+t^2\mu_1}}
	= \frac{\theta\, t\, \mu_1 }{\sqrt{(1+\theta^2\mu_1)^3}}. 
	$$
	This implies 
	$$
	\|(\Onebf+t^2\out{\etabf}{\etabf})^{-1/2}-\Onebf\|_2
	\leq t^2\mu_1
	= t^2\, \| \out{\etabf}{\etabf}\|_2
	\leq t^2\, \|\etabf\|_H^2.
	$$
	Thus, the assertion holds true.\hfill\qed
%\end{proof}	
%
%
%=======================================================================
\subsection{Proof of Proposition~\ref{prop:retr_est_qR}}\label{app:proof:qR}
%\begin{proposition}\label{prop:retr_est_qR}
%The retraction $\calR$ in~\eqref{eq:retr_qR} satisfies
%\[
%	\|\calR(\phibf, t\etabf)-(\phibf+t\etabf)\|_{a_\phibf}\leq 
%	\frac{t^2}{\sqrt{2}}\,\|\phibf+ t\etabf\|_{a_\phibf}\bigl(1+t^2\|\etabf\|_H^2\bigr)^{1/2}\|\etabf\|_H^2.
%\]
%\end{proposition}
%
%\begin{proof}[Proof of Proposition~\ref{prop:retr_est_qR}]
	For a~curve $\varphibf(t)=\phibf+t\etabf$, consider the $qR$~decomposition $\varphibf(t)=\qbf(t)R(t)$. Then we have
	\begin{align}
	\|\calR(\phibf, t\etabf)-(\phibf+t\etabf)\|_{a_\phibf} & = \|\qbf(t)-\qbf(t)R(t)\|_{a_\phibf} 
	\leq \|\qbf(t)\|_{a_\phibf}\|R(0)-R(t)\|_2 \nonumber\\
	& \leq \|\phibf+t\etabf\|_{a_\phibf}\|R^{-1}(t)\|_2\int_0^t\|\dot{R}(s)\|_F\,{\rm d}s. \label{eq:intRdot}
	\end{align}
	Differentiating the relation 
	\begin{equation}\label{eq:RTR}
	\Onebf+t^2\out{\etabf}{\etabf} 
	= \out{\phibf+t\etabf}{\phibf+t\etabf} 
	= R^T(t)R(t),
	\end{equation}
	we obtain 
	$$
	2t\,\out{\etabf}{\etabf} = \dot{R}^T(t)R(t)+R^T(t)\dot{R}(t).
	$$
	Multiplying this equation by $R^{-T}(t)$ and $R^{-1}(t)$ from the left and right, respectively, yields
	$$
	\bigl(\dot{R}(t)R^{-1}(t)\bigr)^T+\dot{R}(t)R^{-1}(t)=2t\,R^{-T}(t)\out{\etabf}{\etabf} R^{-1}(t).
	$$
	Since $\dot{R}(t)R^{-1}(t)$ is upper triangular, we obtain
	$$
	\dot{R}(t)=2t\,\mathrm{up}\bigl(R^{-T}(t)\out{\etabf}{\etabf} R^{-1}(t)\bigr)R(t),
	$$
	where 
	$$
	\bigl(\mathrm{up}(M)\bigr)_{ij}=\left\{\begin{array}{cl}
	M_{ij},   & \text{ if } 1\leq i<j\leq N, \\
	\tfrac{1}{2}M_{ij}, & \text{ if } 1\leq i=j\leq N, \\
	0,        & \text{ otherwise}
	\end{array}\right.
	$$
	for any symmetric matrix $M=[M_{ij}]\in\mathbb{R}^{N\times N}$.  
	As before, let $\mu_1\geq\ldots\geq\mu_N>0$ denote the eigenvalues of~$\out{\etabf}{\etabf}$. 
	Using $2\,\|\mathrm{up}(M)\|_F^2\leq\|M\|_F^2$ and \eqref{eq:RTR}, we have
	\begin{align*}
	2\,\big\|\mathrm{up}\bigl(R^{-T}(t) \out{\etabf}{\etabf} R^{-1}(t)\bigr)\big\|_F^2 
	& \leq \big\|R^{-T}(t) \out{\etabf}{\etabf} R^{-1}(t)\big\|_F^2\\
	& = \tr\bigl(\out{\etabf}{\etabf} R^{-1}(t)R^{-T}(t)\out{\etabf}{\etabf} R^{-1}(t)R^{-T}(t)\bigr)\\
	%& = \tr\bigl(\out{\etabf}{\etabf}(\Onebf+t^2\out{\etabf}{\etabf})^{-1}\out{\etabf}{\etabf}(\Onebf+t^2\out{\etabf}{\etabf})^{-1}\bigr)\\
	& = \big\|\out{\etabf}{\etabf}(\Onebf+t^2\out{\etabf}{\etabf})^{-1}\big\|_F^2 \\
	& = \sum_{i=1}^N\left(\frac{\mu_i}{1+t^2\mu_i}\right)^2 
	\leq \sum_{i=1}^N \mu_i^2 = \big\|\out{\etabf}{\etabf} \big\|_F^2\leq\|\etabf\|_H^4
	\end{align*}
	and, hence, 
	\begin{equation} \label{eq:dotR}
	\|\dot{R}(t)\|_F\leq 2\,t\,\|\mathrm{up}\bigl(R^{-T}(t)\out{\etabf}{\etabf} R^{-1}(t)\bigr)\|_F\|R(t)\|_2
	\leq\sqrt{2}\,t\,\|\etabf\|_H^2\|R(t)\|_2.
	\end{equation}
	Furthermore, \eqref{eq:RTR} implies that
	\begin{align}
	\|R(t)\|_2     & = \|\Onebf+t^2\out{\etabf}{\etabf}\|_2^{1/2}       \leq (1+t^2\|\etabf\|_H^2)^{1/2}, \label{eq:normR}\\
	\|R^{-1}(t)\|_2 & = \|(\Onebf+t^2\out{\etabf}{\etabf})^{-1}\|_2^{1/2} <1.
	\label{eq:normRinv}
	\end{align}
	Thus, the claimed estimate follows from \eqref{eq:intRdot}, \eqref{eq:dotR}, \eqref{eq:normR}, and \eqref{eq:normRinv}. \hfill\qed
%\end{proof}

\end{document}